\documentclass[10pt,reqno]{amsart}
\usepackage{times}
\usepackage[T1]{fontenc}
\usepackage{amssymb}
\usepackage{amsmath,amsthm}
\usepackage{amsfonts}
\usepackage{leftidx}
\usepackage{mathrsfs}
\usepackage{enumerate}	
\usepackage{abstract}
\usepackage{stmaryrd}
\usepackage{ulem}
\usepackage{slashed}
\usepackage{graphicx}
 \usepackage{hyperref}

\usepackage[top=0.8in, bottom=0.8in, left=0.8in, right=0.8in]{geometry}

\def\R{\mathbb{R}}
\def\S{\mathbb{S}}
\def\Z{\mathbb{Z}}

\def\d{|\nabla|}

\def\p{\partial}

\def\vo{\vspace{1\baselineskip}}

\def\be{\begin{equation}}
\def\ee{\end{equation}}

\newtheorem{theorem}{Theorem}[section]

\newtheorem{lemma}{Lemma}[section]
\newtheorem{proposition}{Proposition}[section]
\theoremstyle{definition}

\theoremstyle{remark}

\numberwithin{equation}{section}
\begin{document}
 \title[3D Relativistic Vlasov-Maxwell]{Global solution of the $3D$  relativistic Vlasov-Maxwell system for the large radial data}
\author{Xuecheng Wang}
\address{YMSC, Tsinghua  University, Beijing, China,  100084
}
\email{xuecheng@tsinghua.edu.cn,\quad xuecheng.wang.work@gmail.com 
}
 
\thanks{}

\maketitle

\begin{abstract}

We prove global existence of  the $3D$ relativistic Vlasov-Maxwell system for a class of arbitrary large regular initial data with spherical symmetry, in which the initial distribution function of particles is assumed to decay fast but polynomially towards infinity.

\end{abstract} 
 
\section{Introduction}
 The  $3D$ relativistic  Vlasov-Maxwell system is one of the   fundamental models in the collisionless plasma physics. It reads as follows, 
\be\label{mainequation}
(\textup{RVM})\qquad \left\{\begin{array}{c}
\p_t f + \hat{v} \cdot \nabla_x f + (E+ \hat{v}\times B)\cdot \nabla_v f =0,\\
\nabla \cdot E = 4 \pi \displaystyle{\int_{\R^3}  f(t, x, v) d v}, \qquad \nabla \cdot B =0, \\
\p_t E = \nabla \times B - 4\pi \displaystyle{\int_{\R^3} f(t, x, v) \hat{v} d v }, \quad \p_t B  =- \nabla\times E,\\
f(0,x,v)=f_0(x,v), \quad E(0,x)=E_0(x), \quad B(0,x)=B_0(x).
\end{array}\right.
\ee
where $f:\R_t\times \R_x^3\times \R_v^3\longrightarrow \R_{+}$ denotes the distribution function of particles, $E, B :\R_t\times \R_x^3 \longrightarrow \R^3 $ denote  the electromagnetic field,  $\hat{v}:=v/\sqrt{1+|v|^2}$. As a result of direct computations, 
we can reduce the Maxwell system into the standard wave equations as follows, 
\be\label{electromagnetic}
\p_t^2 B- \Delta B =  { -4\pi\int_{\R^3} \hat v\times \nabla_x f(t, x, v) d v},\quad \p_t^2 E- \Delta E = - 4\pi \int_{\R^3}  \hat{v} \p_t f(t, x, v)   d v- 4\pi \int_{\R^3} \nabla_x f(t, x , v) d v.
\ee 
The following conservation law holds for the RVM system,
 \be\label{conservationlaw}
\mathcal{H}(t):= \int_{\R^3}|E^2(t,x)| + |B^2(t,x)| d x + 8\pi  \int_{\R^3} \int_{\R^3} \sqrt{1+|v|^2} f(t,x,v)  d v = \mathcal{H}(0)
\ee
Moreover for any $p\in [1, \infty]$, we have
 \be\label{conservationlaw2}
\| f(t,x,v)\|_{L^p_{x,v}} = \| f(0,x,v)\|_{L^p_{x,v}}.
\ee
 Note that the  following system of equations satisfied by the backward characteristics, 
 \be\label{characteristicseqn}
\left\{\begin{array}{l}
\displaystyle{\frac{d}{d s} X(s;t, {x}_0, {v}_0 ) = \widehat{V} (s;t,{x}_0, {v}_0 ) } \\ 
\displaystyle{\frac{d}{d s} V(s;t,{x}_0, {v}_0 ) = K(s, X(s;t,{x}_0, {v}_0 ), V(s;t, {x}_0, {v}_0))  } \\ 
 X(t;t,{x}_0, {v}_0 )= {x}_0,  V(t;t,{x}_0, {v}_0 )= {v}_0,\\
\end{array}\right. 
\ee 
 where $K(s,x,v):=E(s,x)+\hat{v}\times B(s,x)$. Due to the transport nature of  the Vlasov equation,  we have $f(t,x_0,v_0)=f(s, X(s;t,{x}_0, {v}_0 ) ,  V(s;t,{x}_0, {v}_0 ) )$ for any $s\in (-T, T)$, where $T$ denotes the maximal time of existence.  

There is a large literature in the study of the   Vlasov-Maxwell system.  Large data global solutions for the RVM have been constructed  in  two dimensions  and two and half dimensions, we refer readers to \cite{glasseys1,luk} and reference therein for more comprehensive introduction and more details. Moreover,  global solutions for the RVM in the three dimensions  for small data have also been constructed, we refer readers to \cite{glassey2,schaeffer1,wang} for more details.

  In this paper,    we are mainly  interested in the large data Cauchy problem of RVM in the  $3D$ case. The  result  of  Glassey-Strauss \cite{glassey3} says that the  classical solution can be globally  extended  as long as the particle density has compact support in $v$ for all the time. A new proof of this result based on Fourier analysis was   given by Klainerman-Staffilani \cite{Klainerman3}, which adds a new perspective to the study of $3D$ RVM system, see also \cite{glasseys2, pallard2,alfonso}.  An interesting line of research is the continuation criterion for the global existence of the Vlasov-Maxwell system. In \cite{glassey4}, Glassey-Strauss showed that the lifespan of the solution of the relativistic Vlasov-Maxwell system can be continued if the initial data decay at rate $|v|^{-7}$ as $|v|\rightarrow \infty$ and   $\| (1+|v|) f(t,x,v)\|_{L^\infty_x L^1_v}$ remains bounded for all time.
 An improvement of this result and a new continuation criterion was given by Luk-Strain \cite{luk}, which says that a regular solution can be extended as long as $\|(1+|v|^2)^{\theta/2} f(t,x,v)\|_{L^q_x L_v^1}$ remains bounded for $\theta > 2/q, 2< q \leq +\infty$, see also Kunze\cite{kunze}, Pallard\cite{pallard3}, and Patel\cite{patel} for the recent improvements  on the continuation criterion. 

We assume that the initial data  $(f_0(x,v), E_0(x), B_0(x))$ are spherically symmetric in the following sense, 
\be\label{radialsym}
f_0(Rx, Rv) = f_0( x,v), \quad  E_0(  Rx)= RE_0(x), \quad B_0( Rx)= RB_0(x),  \quad \forall R\in SO(3). 
\ee
As a result of direct computation and the uniqueness of solution, we know that the above radial   symmetry property can be propagated from the initial data.  Our main result in this paper is summarized as follows.
 \begin{theorem}\label{maintheorem}
Assume that  initial data $(f_0, E_0, B_0)$ are radial  \textup{(}in the sense of \textup{(\ref{radialsym})}\textup{)}$,  f_0(x,v)\in H^s(\R_x^3 \times \R_v^3)$, $E_0, B_0\in H^s(\R^3)$, $s\in \mathbb{Z}_{+},s\geq 6$. Moreover,  for  $N_0:= 10^{8}$,  we assume that the following estimate holds for the initial distribution function $f_0(x,v),$
\be\label{assumptiononinitialdata}
\sum_{\alpha \in \mathbb{Z}_{+}^6,|\alpha|\leq s} \|(1+|x|+|v|)^{N_0}\nabla_{x,v}^\alpha f_0(x,v)\|_{L^2_{x,v}}  < +\infty. 
\ee
Then the relativistic Vlasov-Maxwell system (\ref{mainequation})   admits global solution $(f(t), E(t), B(t))$ in $ H^s(\R_x^3 \times \R_v^3)\times H^s(\R^3_x)\times H^s(\R^3_x)$. 

\end{theorem}

\subsection{Local theory and  the main idea  of proof}

Since the assumption imposed on the initial data in (\ref{assumptiononinitialdata}) is stronger than the assumptions required for initial data in the work of Luk-Strain \cite{luk},  the continuation criteria obtained there can be applied directly in this paper. From  the work of Luk-Strain \cite{luk}, we know that the $3D$ relativistic  Vlasov-Maxwell system is local well-posed and the regularity can be propagated within the  time interval of existence $[0,T)$, where $T$ denotes the maximal time of existence.   Moreover, from the work of    Luk-Strain \cite{luk},  we know that the lifespan can be extended to $[0,T+\epsilon]$ for some positive number $\epsilon$ if the following assumption holds, 
\be\label{march13eqn21}
\sup_{t\in[0, T)} \sup_{x_0,v_0\in \R^3} \int_0^T\big(|E(s,X(s;t,x_0,v_0))| +|B(s,X(s;t,x_0,v_0))|  \big)  ds < +\infty. 
\ee
Readers are refereed to  \cite{luk}[Theorem 5.7] for more details.

The main goal of this paper is to show that the acceleration accumulated along characteristics is indeed uniformly bounded for all time. Hence finishing the proof of global  regularity for the $3D$ relativistic RVM system. 

To this end, we propagate a high order moment of the distribution function of particles. By using the Glassey-Strauss decomposition and the spherical symmetry of the solution, we show that the quantity in (\ref{march13eqn21}) is controlled from the above by the moment of the distribution function, see the estimate (\ref{feb25eqn72}) in Lemma \ref{roughalongchar2}. 

Actually, we not only  show that the boundedness of the   high order moment  but also show that it grows at most polynomially over time, see the estimates (\ref{march21eqn21}) and (\ref{march21eqn22}). To this end,  as summarized in the Proposition \ref{mainprop}, the main observation is that   the majority of   particles, which are localized around zero due to the polynomial decay assumption on the initial data,  will not be accelerated much after the speed of particles reaches a certain level. 

Intuitively speaking,  particles will travel toward infinity when  the speed of particles reaches a certain level. Due to the spherical symmetry, the electromagnetic field is localized around zero and it becomes weaker as the radius becomes larger. As a result, the accumulate acceleration caused by the electromagnetic field  is not strong.  To get more transparent intuition,  we refer readers to \cite{wang2} for a similar result in the $3D$  relativistic Vlasov-Poisson system, which is a simpler model of RVM.

To prove Proposition \ref{mainprop}, we measure carefully the gain and the loss of using the smoothing effects. One type of the smoothing effects comes from the oscillation in time for the electromagnetic field itself. The other   type of the smoothing effect, as pointed out by  Klainerman-Staffilani \cite{Klainerman3}, is that the integration of electromagnetic field along the characteristic is smoother.  

This paper is organized as follows.

\begin{enumerate}
	\item[$\bullet$] In section \ref{pre}, we introduce the notation and prove two basic lemmas used in this paper.
	\item[$\bullet$] In section \ref{proof},   we introduce the set-up of propagation of moment and prove the Theorem \ref{maintheorem}  under the assumption that we have a good control of the electromagnetic field and the increment of characteristics in terms of the moment of the distribution function. 
\item[$\bullet$] In section \ref{roughcontrol},  we give a rough control of the electromagnetic field based on the Glassey-Strauss decomposition. 
\item[$\bullet$] In section \ref{incrementestimate}, by exploiting the smoothing effect, we use a Fourier method to control the increment of the magnitudes of the spatial characteristic $X(s;t,x_0,v_0)$ and the velocity characteristic   $V(s;t,x_0,v_0)$ over time. 
\end{enumerate}

\vo

\noindent \textbf{Acknowledgment}\qquad The author is supported by NSFC-11801299.

\section{Preliminary}\label{pre}
 
 For any two numbers $A$ and $B$, we use  $A\lesssim B$ and $B\gtrsim A$  to denote  $A\leq C B$, where $C$ is an absolute constant. We use the convention that all constants which only depend on the initial data, e.g., the conserved quantities  ($\|f(t,x,v)\|_{L^p_{x,v}},$ $p\in[1,\infty]$, $\|(1+|v|^2)f(t,x,v)\|_{L^1_{x,v}},$ $ \|E(t)\|_{L^2}, \|B(t)\|_{L^2}$), will be treated as absolute constants. 

 For any two vectors $v,u\in \R^3$, we use $\angle(v,u)$ to denote the angle between $v$ and $u$ and use the convention that $\angle(v,u)\in[0,\pi]$. For any $r\in \R_{+}$, we use  $r_{-}$ to denote $\min\{r,1\}.$ For any vector $v\in \R^3/\{0\}$, we use $\tilde{v}:=v/|v|\in \mathbb{S}^2$ to denotes the direction of $v$.  Note that, for any $v\in\R^3/\{0\}$,  we define $S_{v}:=\{\omega,\omega\in \mathbb{S}^2,\omega\cdot \tilde{v} =0\}$ to be the great circle on sphere that is orthogonal to the direction $\tilde{v}$.

  We  fix an even smooth function $\tilde{\psi}:\R \rightarrow [0,1]$, which is supported in $[-3/2,3/2]$ and equals to ``$1$'' in $[-5/4, 5/4]$. For any $k\in \mathbb{Z}$, we define the cutoff functions $\psi_k, \psi_{\leq k}, \psi_{\geq k}:\cup_{n=1,3}\R^n\longrightarrow \R$ as follows, 
\[
\psi_{k}(x) := \tilde{\psi}(|x|/2^k) -\tilde{\psi}(|x|/2^{k-1}), \quad \psi_{\leq k}(x):= \tilde{\psi}(|x|/2^k)=\sum_{l\leq k}\psi_{l}(x), \quad \psi_{\geq k}(x):= 1-\psi_{\leq k-1}(x).
\]
 Moreover, for any $l,n\in \Z, l\geq n$, we define   the  cutoff function $\varphi_{l;	n}(\cdot)$ with threshold $n$ as follows, 
 \be\label{cutoffwith}
\varphi_{l;	n}(x) =\left\{\begin{array}{ll}
\psi_{\leq n}(x) & \textup{if\,\,} l=n\\
\psi_{l}(x)  & \textup{if\,\,} l> n.\\
\end{array}\right. 
 \ee
If $n=0$, then we use the convention that $\varphi_j(\cdot)$ denotes $\varphi_{j;0}(\cdot).$

We first record  the classic Kirchhoff's formula, which allows us to represent the solution of linear wave in physical space.  
 
\begin{lemma}\label{Kirchoff}
 For any $t\in \R,x\in \R^3$, the following equality holds, 
\be\label{dec20eqn1}
\d^{-1}   \sin( t\d ) h( x)  = \frac{1}{4\pi}   t  \int_{\mathbb{S}^2} h(x+t \theta) d \theta,
\ee
\be\label{march4eqn100}
\d^{-1}   \cos(t\d) h(x)= \frac{1}{4\pi}     \int_{\mathbb{S}^2}  \d^{-1}  h(x+t \theta) d \theta     + \frac{1}{4\pi}     \int_{\mathbb{S}^2} t \theta\cdot \frac{\nabla}{\d} h(x+t \theta) d \theta. 
\ee
\end{lemma}
\begin{proof}
Note that 
\be\label{march4eqn41}
 \int_{\R^3} e^{-i   x\cdot \xi } \int_{\mathbb{S}^2} h(x+t\theta) d \theta d x = \int_{\mathbb{S}^2 } e^{i  t \xi \cdot \theta} \hat{h}(\xi) d\theta=  2\pi \hat{h}(\xi)  \int_{0}^{\pi}e^{it|\xi|\cos(\phi)} \sin(\phi) d \phi = 	\frac{4\pi \sin(  t |\xi|)}{t|\xi|}  \hat{h}(\xi)  .
\ee
 Hence finishing the proof of the desired formula (\ref{dec20eqn1}). Our desired equality (\ref{march4eqn100}) holds after taking derivative with respect to ``$t$'' for  the equality (\ref{dec20eqn1}).
\end{proof}

Recall the equations satisfied by the electromagnetic field in (\ref{electromagnetic}) and   the Kirchhoff's formula in (\ref{dec20eqn1}).  From the Duhamel's formula, the following decomposition holds after we do dyadic decomposition for the velocity variable, 
\be\label{march14eqn1}
K(t)= K_{free}(t) +  \sum_{j\in \mathbb{Z}_{+}} K_j(t), \quad K\in\{E, B\}.
\ee
where $K_{free}(t), K\in \{E, B\},$   denote  the linear wave solution determined by the initial data of RVM (\ref{mainequation}), 
\be\label{march14eqn2}
K_{free}(t)=\cos(t\d)K_0 + \sin(t\d)\d^{-1}(\p_t K)\big|_{t=0}, \quad K\in \{E, B\},
\ee
\be\label{march14eqn4}
E_j(t):=     -   \int_{0}^{t}  \int_{\R^3} \int_{\mathbb{S}^2}  (t-s)   \big(  \hat{v} \p_t f(s, x+(t-s)\theta, v)    + \nabla_x f(s, x+(t-s)\theta, v)\big) \varphi_j(v)d\theta  d v d s,
\ee
\be\label{march14eqn5}
B_j(t):= - \int_{0}^{t}   \int_{\mathbb{S}^2}\int_{\R^3}  (t-s) \hat v\times \nabla_x f(s, x+(t-s)\theta, v) \varphi_j(v) d\theta  d v d s.
\ee

The benefit of radial symmetry is mainly exploited in the following Lemma. Essentially speaking, it says that 	the average on sphere for an integrable radial function is well controlled if the center of sphere is far away from zero. Also, it's very natural that the radial symmetry won't provide any gain when the center is close to zero, which explains why the estimates (\ref{march6eqn1}) and (\ref{feb29eqn1}) behaves badly when $r$ approaches to zero. 
 \begin{lemma}\label{basicestimateint}
For any fixed $\omega_0 \in \S^2$,  any radial function $h:\R^3\longrightarrow \mathbb{C}$, the following estimate holds for any $ l\in \mathbb{Z}\cap(-\infty,2],   x\in \R^3/\{0\}, s\in \R_{+}/\{|x|\}$, 
\be\label{march6eqn1}
\big| \int_{\mathbb{S}^2} h(x +s\omega) \psi_{\leq l}(\angle(\omega, \omega_0))  d \omega  \big| \lesssim \frac{2^l}{(2^l+| \tilde{x}\times \omega_0|) rs} \min\{ \frac{1}{ |r-s|} \|h\|_{L^1_x},\sup_{z\in \R^3, |z|\in [|r-s|, r+s]} s|z||h(z)| \}, 
\ee
where $\tilde{x}:=x/|x|$ and $r:=|x|$.
Moreover,  for  any fixed   $a,b\in \R, $   and  any radial function $f:\R_x^3\times \R_v^3 \longrightarrow \R_{+} $ in the sense of \textup{(\ref{radialsym})}, the following estimate holds,
  \be\label{feb29eqn1}
 \big| \int_{\R^3}\int_{\mathbb{S}^2}\int_{S_v}  f(x-a \tilde{v} -b\omega_{v}+ s \omega, v)  \psi_{\leq l}(\angle(\omega, \omega_0)) d\omega_v   d \omega d v\big|\lesssim  \frac{ 2^l\| f(x,v)\|_{L^1_{x,v}}}{(2^l+|\tilde{x}\times \omega_0|)r s |r-s|}. 
 \ee 
 where   for  any fixed $v\in \R^3/\{0\}, S_v:=\{\theta\in \S^2, \theta\cdot \tilde{v}=0\}$.
\end{lemma}
\begin{proof}
Since $h$ is radial and fixed $\omega_0$ is arbitrary, without loss of generality, we assume that $x=(r,0,0), r:=|x|$. Let $\omega:=(\cos\theta, \sin\theta \cos\phi, \sin\theta \sin \phi )$ and  $\omega_0:=(\cos\theta_0, \sin\theta_0 \cos\phi_0, \sin\theta_0 \sin \phi_0 )$. Due to the cutoff function $\psi_{\leq l}(\angle(\omega, \omega_0))$, we know that
\[
|supp_{\phi}(\psi_{\leq l}(\angle(\omega, \omega_0)))| \lesssim \frac{2^l}{2^l + \sin \theta_0} \sim \frac{2^l}{2^l + |\tilde{x}\times \omega_0|}=:C(x,\omega_0).
\] 
Therefore, from the above estimate and the radial symmetry of $h$, we have 
\[
\big| \int_{\mathbb{S}^2} h(x +s\omega) \psi_{\leq l}(\angle(\omega, \omega_0))  d \omega    \big| \lesssim  C(x,\omega_0)  \int_0^\pi |h((\sqrt{r^2+s^2 + 2rs \cos \theta},0,0))| \sin \theta d \theta  
\]
\be\label{march7eqn1}
 \lesssim  \frac{C(x,\omega_0)}{  rs} \int_{|r-s|}^{r+s} |h((z,0,0))| z d z\lesssim   \frac{C(x,\omega_0)}{  rs}\min\{\frac{1}{|r-s|} \int_{|r-s|}^{r+s} |h((z,0,0))| z^2 d z,  \sup_{y\in \R^3, |y|\in [|r-s|, r+s]} s|y||h(y)|  \}.
\ee
Hence finishing the proof of our desired estimate (\ref{march6eqn1}).
In the above estimate, we used the change coordinates $\theta\longrightarrow z:=\sqrt{r^2+s^2+2rs \cos\theta}$.   For fixed $a,b\in \R$, we  define
\[
\rho(x   ;a,b):=\int_{\R^3} \int_{S_v}  f(x-a \tilde{v} -b\omega_{v}, v)   d\omega_v   d v.
\]
Note that, $\forall  R\in SO(3),$ the following equality holds from the radial symmetry of $f(x,v)$ in (\ref{radialsym}),
\[
\rho( R x ;a,b )=  \int_{\R^3}  \int_{S_v}  f(Rx-a \tilde{v} -b\omega_{v}, v)   d\omega_v   d v=  \int_{\R^3} 	 \int_{\R^3}  \int_{S_u}  f(Rx-a R\tilde{u} -bR\omega_{u}, Ru)   d\omega_u   d v
\]
\[ 
 =   \int_{\R^3} \int_{S_u}  f(x-a \tilde{u} -b\omega_{u}, u)   d\omega_u   d v=\rho(x;a,b) .
\]
Therefore,   our desired estimate (\ref{feb29eqn1}) holds directly after we applying the estimate (\ref{march6eqn1})  to the radial function $\rho(x;a,b)$.  
 
\end{proof}

\section{Propagation of moments and the proof of Theorem}\label{proof}

Let $N_1= N_0\times 10^{-3}=10^5$. Define
\be\label{jan18eqn1}
 M_{N_1}(t):= \int_{\R^3} \int_{\R^3} (1+|v|)^{N_1} f(t,x,v) d x d v ,\quad 
\tilde{M}_{N_1}(t):= (1+t)^{ N_1^2}+ \sup_{s\in [0,t]}  M_{N_1}(s).
\ee
Note that for any $t\in [0, T ), j\in \mathbb{Z}_{+}$, we have
\[
\int_{\R^3}  f(t,x,v)  \psi_j(v) d v  \leq    \min\{2^{3j},\int_{\R^3}  f(t,x,v)  \psi_j(v) d v \} \lesssim   2^{3j(1-1/p)}\big( \int_{\R^3}  f(t,x,v)  \psi_j(v) d v\big)^{1/p}
\]
Hence, we obtain the following basic estimate for any $p\in [1,\infty], $
\be\label{feb25eqn12}
\|   f(t,x,v)  \psi_j(v) \|_{ L^p_x L^1_v }\lesssim 2^{3j(1-1/p)} \big( \int_{\R^3} \int_{\R^3} f(t,x,v)  \psi_j(v) d v\big)^{1/p} 
\lesssim  2^{3j(1-1/p)} \big( \min\{2^{-j}, 2^{- N_1 j}  M_{N_1}(t)\} \big)^{1/p}.
\ee

 Let $M_t$ denotes the minimum integer such that $2^{M_t}\geq (\tilde{M}_{N_1}(t))^{1/ {N_1}}$. Hence $2^{M_t}   \sim(\tilde{M}_{N_1}(t))^{1/{N_1}}$.  We define a set of majorities of particles at time $s$ as follows,  
\be\label{march19eqn31}
R  (t,s):= \{(X(s;t, x_0,v_0),V(s;t,x_0,v_0)) :  |X(0;t,x_0,v_0)|+|V(0;t, x_0,v_0))|\leq  2^{\beta M_t}, x_0,v_0\in \R^3 \}, \quad \beta:=1/300.
\ee

 Let $t\in [0,T)$ and the initial data $(x_0, v_0)\in R(t,0)$ of characteristics in (\ref{characteristicseqn}) be fixed. For the simplicity of notation, we will omit the dependence of characteristics with respect to the initial data and view the spatial characteristic $X(s)$ and velocity characteristic  $V(s)$ as regular functions with respect to time $s$.

We will use a standard bootstrap argument to show that the size of any   velocity characteristic $V(s)$, which starts from the major set $R(t,0)$,  is uniformly bounded by $  2^{(1-\beta)M}$, $\beta:=1/300$, for all $s\in [0,t]$. More precisely, we have 
 \begin{proposition}\label{mainprop}
For any $t\in[0, T)$, the following relation holds for some sufficiently large absolute constant $C$, 
\be\label{jan13eqn11}
R(t,t)\subset B(0,   C 2^{ \beta {M_t} })\times B(0,  C 2^{(1-\beta)M_t}).  
\ee

\end{proposition}
\begin{proof}

Recall (\ref{march19eqn31}) and the system of equations satisfied by characteristics in (\ref{characteristicseqn}). First of all, for any $t_1, t_2\in  [0,t]$,  the following rough estimate holds, 
\be\label{march17eqn11}
||X(t_2)| -|X(t_1)| |\leq \int_{t_1}^{t_2} |\tilde{X}(s)\cdot \hat{V}(s)| ds \leq  |t_2-t_1|, \quad \Longrightarrow \sup_{s\in[0,t]} |X(s)| \leq 2^{\beta M_t} +|t| <  2^{\beta M_t+1}.
\ee
Hence finishing the proof of first part of (\ref{jan13eqn11}).  

 From the continuity of characteristics and the size of initial data in $R(t,0)$, we know that one of the following two scenarios holds. 
\begin{enumerate}
\item[(i)] There exists    a maximal time  $T^{\ast}\in (0,t)$  and a maximal number  $T^{\ast\ast} > T^{\ast}$ such that the following estimates holds, 
\be\label{aproiriestimate}
\sup_{s   \in [0, T^{\ast}]}   |V(s )| \leq  2^{(1-\beta) M_t-1},  \quad |V(T^{\ast})|= 2^{(1-\beta) M_t-1},\quad 
\ee
\be\label{march12eqn31}
2^{(1-\beta) M_t-2} \leq \inf_{s \in [T^{\ast},  T^{\ast\ast} ]}   |V(s)|\leq \sup_{s \in [T^{\ast},  T^{\ast\ast} ]}  |V(s)| \leq   2^{(1-\beta)M_t}, \quad [T^{\ast},  T^{\ast\ast} ]\subset[0,t].
\ee
\item[(ii)] The following estimate holds, 
\be\label{aproiriestimatetrivial}
\sup_{s   \in [0, t]}   |V(s )| \leq   2^{(1-\beta)M_t-1} . 
\ee
\end{enumerate}

If the second scenario happens or $\tilde{M}_t \lesssim 1$, then there is nothing left to be proved, we restrict ourself to the first scenario under the assumption that  $\tilde{M}_t \gg 1$. We will show that the estimate (\ref{march12eqn31}) can be improved hence close the bootstrap argument. To obtain an improved estimate, we control the increment of velocity between any two time in $[T^{\ast},  T^{\ast\ast} ]\subset[0,t]$, which is also the main mission of this paper.

As a result of direct computations, we have
\be\label{jan12eqn21}
\frac{d}{d s} | X(s )|^2   = 2X(s ) \cdot  \hat{V} (s ), \quad \frac{d}{d s} | V (s )|    = 2  \tilde{V}(s )\cdot E(s, X(s)).
\ee
To better see the dynamics of the magnitude of spatial characteristic, we also study the second order derivative of $|X(s)|^2$. As a result, we have 
\[
\frac{d}{d s}\big( X(s  ) \cdot   \hat{V} (s )\big)   = |\hat{V}(s  ) |  +  \frac{X(s  ) \cdot \big(E(s, X(s  ) ) + \hat{V}(s   ) \times B(s,X(s  ) )\big)}{{\sqrt{1+ | V(s)|^2 }}} - \frac{X(s)\cdot {V}(s ) }{\big( {1+ | V(s )|^2 }\big)^{3/2}} \big({V}(s ) \cdot  E(s,X(s))\big)
\]
\be\label{march1eqn2100}
=|\hat{V}(s)| + C_1(X(s), V(s))\cdot E(s, X(s)) +C_2(X(s), V(s))\cdot B(s, X(s))
\ee
where
\be\label{march17eqn1}
 C_1(X(s), V(s)):= \frac{X(s)}{\sqrt{1+|V(s)|^2}}- \frac{X(s)\cdot V(s)}{(1+|V(s)|^2)^{3/2}} V(s), \quad C_2(X(s), V(s)):=\frac{X(s)\times \hat{V}(s)}{\sqrt{1+|V(s)|^2} }. 
\ee
From the estimate (\ref{march16eqn51}) in Proposition \ref{mainproposition}, the following rough estimate holds for any $t_1, t_2\in  [T^{\ast}, T^{\ast\ast} ]\subset[0,t]$, s.t., $t_1\leq t_2,$
\[
  X(t_2 ) \cdot   \hat{V} (t_2 ) - X(t_1 ) \cdot   \hat{V} (t_1 )   \geq   \frac{4}{5} (t_2-t_1) -  \big|\int_{t_1}^{t_2} C_1(X(s),V(s)) \cdot E(s, X(s  ) ) d s \big|
\]	
\be\label{march17eqn21}
-\big|\int_{t_1}^{t_2}(  C_2(X(s), V(s)) \cdot  B(s,X(s  ) )  d s\big| \geq \frac{2}{3}(t_2-t_1) - 2^{-2M_t +16\alpha M_t},\quad \alpha:=1/100=3\beta.
\ee
 \noindent \textbf{Case $1$:\quad}  If there exists a time $\tau_{\star}\in [T^{\ast}, T^{\ast\ast} ]$ s.t.,  $X(\tau_{\star})\cdot \hat{V}(\tau_{\star})=0$. 

From the estimate (\ref{march17eqn21}), the following estimate holds for $s_1,s_2\geq 0$ s.t, $s_1+\tau_{\star}, -s_2+\tau_{\star}\in[T^{\ast}, T^{\ast\ast} ]$,
\[
|X(\tau_{\star})|^2-|X(\tau_{\star}-s_2)|^2= \int_{0}^{s_2} 2X(\tau_{\star}-s_2+\tau)\cdot \hat{V}(\tau_{\star}-s_2+\tau)d  \tau \leq  \int_{0}^{s_2} - \frac{4}{3}(s_2-\tau) +  2^{-2M_t+16\alpha M_t} d \tau 
\]
\be\label{march17eqn32}
 \Longrightarrow |X(\tau_{\star}-s_2)|^2\geq |X(\tau_{\star})|^2 + \frac{2}{3} s_2^2 - 2^{-2M_t+16\alpha M_t} s_2. 
\ee
\be\label{march17eqn31}
|X(\tau_{\star}+s_1)|^2 -|X(\tau_{\star})|^2 \geq\int_{0}^{s_1} 2\big(\frac{2}{3}\tau - 2^{-2M_t+7\alpha  M_t} \big) d \tau \gtrsim s_1^2-  2^{-2M_t+16\alpha M_t}s_1. 
\ee

 \noindent \textbf{Case $2$:\quad}  If there doesn't exist  a time $\tau_{\star}\in [T^{\ast},  T^{\ast\ast} ]$ s.t.,  $X(\tau_{\star})\cdot \hat{V}(\tau_{\star})=0$.

For this case we know that $|X(s)|^2$ is a monotonic function with respect to $s$, which implies that either   $|X(T^{\ast}   )|^2$   or   $|X(T^{\ast\ast} )|^2$   is the minimum. If $|X( T^{\ast\ast} )|^2$ is the minimum , i.e., $X(s) \cdot \hat{V}( s)\leq 0$, then the following estimate holds from (\ref{march17eqn21}), 
 \[
 - X( T^{\ast\ast} -s ) \cdot   \hat{V} ( T^{\ast\ast} -s)\geq   X( T^{\ast\ast}  ) \cdot   \hat{V} (  T^{\ast\ast} ) - X(  T^{\ast\ast} -s ) \cdot   \hat{V} ( T^{\ast\ast} -s)   \geq \frac{2}{3} s -  2^{-2M_t+16 \alpha M_t}  . 
 \]
 \[
 \Longrightarrow |X(   T^{\ast\ast}  )|^2-|X(   T^{\ast\ast} -\tau )|^2 = \int_{0}^{\tau} 2 X(  T^{\ast\ast} -s)\cdot \hat{V}(  T^{\ast\ast} -s) d s \leq  -\frac{2}{3}\tau^2 +  2^{-2M_t+16\alpha  M_t+1} \tau
 \]
 \be\label{march17eqn33}
 \Longrightarrow |X(  T^{\ast\ast} -\tau )|^2 \geq |X( T^{\ast\ast}  )|^2+  \frac{2}{3} \tau^2 -  2^{-2M_t+16 \alpha M_t+1} \tau, \quad \tau\in [0, T^{\ast\ast}-T^{\ast}]. 
 \ee
If  $|X( T^{\ast } )|^2$ is the minimum , i.e., $X(s) \cdot \hat{V}( s)\geq  0$, then the following estimate holds from (\ref{march17eqn21}), 
\be\label{march25eqn1}
|X(   T^{\ast } +\tau  )|^2-|X(   T^{\ast } )|^2 = \int_{0}^{\tau} 2 X(  T^{\ast  } +s)\cdot \hat{V}(  T^{\ast } +s) d s\geq  \frac{2}{3} \tau^2 -  2^{-2M_t+16 \alpha M_t+1} \tau, \quad \tau\in [0, T^{\ast\ast}-T^{\ast}].  
\ee

 To sum up, from the estimates (\ref{march17eqn32}--\ref{march25eqn1}), we have a good control of  the magnitude of characteristic for all time except a small neighborhood of the fixed local (global) minimum. 

 Since the   \textbf{Case} $2$ is of the same type as in the  \textbf{Case} $1$, 
  without loss of generality,  we restrict ourself to the   \textbf{Case} $1$. Let $\tau_0:=2^{-2M_t+20\alpha M_t}$. Based on the possible size of $|X(\tau_{\star})|$, we separate into two cases as follows.  

 \noindent $\oplus$\quad If $|X(\tau_{\star})|\geq 2^{-5M_t/3}$. 

  From the rough estimate of increment of magnitude of spatial characteristic in (\ref{march17eqn11}), we have
 \be\label{march17eqn47}
\sup_{s\in [\tau_{\star}- \tau_0, \tau_{\star}+ \tau_0]\cap [T^{\ast}, T^{\ast\ast} ]} |X(s)|\geq 2^{-5M_t/3}-2^{-2M_t+20\alpha M_t}\gtrsim  2^{-5M_t/3}.
 \ee
 From the above estimate,  the estimate (\ref{march14eqn41}) in Proposition \ref{mainproposition} and the rough estimate of the electromagnetic field (\ref{march3eqn51}) in Lemma \ref{roughestimateofelectromag}, the following estimate holds for any $\kappa\in[-2\tau_0, 2\tau_0  ] $,   s.t., $\tau_{\star}+\kappa\in  [T^{\ast}, T^{\ast\ast} ],$
 \be\label{march17eqn81}
 \big||V(\tau_{\star}+\kappa)| -  |V(\tau_{\star})| \big|\lesssim2^{M_t/3+7\alpha M_t} +\big|  \int_{0}^{\kappa} \frac{2^{(1-\alpha+\epsilon)M_t}}{|X(\tau_{\star}+s)|_{-}} ds 
\big| \lesssim  2^{2M_t/3+20\alpha M_t}.
 \ee
 For any $\kappa\in[ -\tau_0, \tau_0 ]^c $, s.t., $\tau_{\star}+\kappa\in  [T^{\ast},  T^{\ast\ast} ],$ from the estimates (\ref{march17eqn32})  and   (\ref{march17eqn31}), we have 
 \[
 |X(\tau_{\star}+\kappa)| \gtrsim |X(\tau_{\star})|+|\kappa|
 \]
 From the above estimate, the rough estimate of the electromagnetic field (\ref{march3eqn51}) in Lemma  \ref{roughestimateofelectromag}  and the estimate (\ref{march14eqn41}) in Proposition \ref{mainproposition}, the following estimate holds for   $\kappa\in  [ -2\tau_0,2 \tau_0 ]^c$, s.t., $\tau_{\star}+\kappa\in  [T^{\ast}, T^{\ast\ast} ],$
 \be\label{march17eqn62}
  \big||V(\tau_{\star}+\kappa)| -  |V(\tau_{\star}+sign(\kappa)\tau_0)| \big|\lesssim 2^{M_t/3+7\alpha M_t} + \big| \int_{\tau_{\star}+ sign(\kappa)\tau_0}^{\tau_{\star}+   \kappa} \frac{2^{(1-\alpha+\epsilon)M_t}}{\min\{|X(\tau_{\star})|+|s-\tau_{\star}|,1\} }  d s \big| \lesssim 2^{(1-\alpha)M_t+2\epsilon M_t}.
 \ee
where $\epsilon:=60/N_1=6\times 10^{-4}.$

  \noindent $\oplus$\quad If $|X(\tau_{\star})|\leq 2^{-5M_t/3}$. 

From the rough estimate of increment of magnitude of spatial characteristic in (\ref{march17eqn11}), the following estimate holds for $s\in  [-\tau_0, \tau_0]$ s.t., $ \tau_{\star}+s \in  [T^{\ast},  T^{\ast\ast} ],$
 \[
 |X(\tau_{\star}+s)|\leq  2^{-5M_t/3} +  s\lesssim  2^{-5M_t/3}. 
 \]
Since now $ |X(\tau_{\star}+s)|, s\in[-\tau_0, \tau_0]$, is small, from   the estimate (\ref{march16eqn51}) in Proposition \ref{mainproposition},  the following improved estimate holds, for any $t_1,t_2\in [\tau_{\star}-\tau_0, \tau_{\star}+\tau_0]\cap  [T^{\ast},  T^{\ast\ast} ]$, s.t., $t_1\leq t_2$, 
\[
  X(t_2 ) \cdot   \hat{V} (t_2 ) - X(t_1 ) \cdot   \hat{V} (t_1 )   \geq  \frac{2}{3}(t_2-t_1) - 2^{-5M_t/3-2M_t/3+9\alpha M_t}. 
\]
With the above improved estimate,  the following improved estimate holds for any $s\in[-\tau_0, \tau_0]$, s.t., $\tau_{\star}+s\in  [T^{\ast},  T^{\ast\ast} ],$
\be\label{march17eqn61}
 \Longrightarrow |X(\tau_{\star}+s)|^2\geq |X(\tau_{\star})|^2 + \frac{2}{3} s^2 -   2^{ -7M_t/3+9\alpha M_t} |s|. 
\ee
Let $\tau_1:=  2^{-7M_t/3+10\alpha M_t}< \tau_0$. Recall the decomposition of electromagnetic field in (\ref{march14eqn1}).  From   the rough estimate of electromagnetic field (\ref{march15eqn35}) in Lemma \ref{roughestimateelectro}, which is used for the case   $j\leq (1+\epsilon) M_t$, and the second estimate in  (\ref{feb25eqn72}) in Lemma 
\ref{roughalongchar2}, which is used for the the case   $j\geq (1+\epsilon) M_t$, the following estimate holds for any $\kappa\in [- 2\tau_1, 2\tau_1]$, s.t., $\tau_{\star}+\kappa\in  [T^{\ast},  T^{\ast\ast} ],$
 \be\label{march17eqn63}
 \big||V(\tau_{\star}+\kappa)|- |V(\tau_{\star} )|\big| \lesssim \sum_{j\in\mathbb{Z}_{+}} \sum_{K\in\{E,B\}} \int_{\tau_{\star}}^{\tau_{\star}+\kappa} |K_j(s, X(s))| d  s \lesssim     2^{-7M_t/3+10\alpha M_t} 2^{3(1+3\epsilon)M_t}  \lesssim 2^{ 2M_t/3+11\alpha M_t}. 
 \ee
 For any $\kappa\in  [- \tau_1, \tau_1 ]^c$, s.t., $\tau_{\star}+\kappa\in  [T^{\ast},  T^{\ast\ast} ],$ from the estimates (\ref{march17eqn31}), (\ref{march17eqn32}), which are used for $\tau\in  [- \tau_0, \tau_0 ]^c$  and  the estimate (\ref{march17eqn61}), which is used when $\tau\in  [- \tau_1, \tau_1 ]^c\cap[- \tau_0, \tau_0 ] $,   we have
 \[
 |X(\tau_{\star}+\kappa)| \gtrsim |X(\tau_{\star})|+|\kappa|.
 \]
 From the above estimate,  the rough estimate of the electromagnetic field (\ref{march3eqn51}) in Lemma  \ref{roughestimateofelectromag} , the estimate (\ref{march14eqn41}) in Proposition \ref{mainproposition}, the following estimate holds for any  $\kappa\in  [ -2\tau_1, 2\tau_1 ]^c$, s.t., $\tau_{\star}+\kappa\in  [T^{\ast}, T^{\ast\ast} ],$
 \be\label{march17eqn64}
  \big||V(\tau_{\star}+\kappa)| -  |V(\tau_{\star}+sign(\kappa)\tau_1 )| \big|\lesssim 2^{M_t/3+7\alpha M_t} +\big| \int_{ \tau_{\star}+sign(\kappa)\tau_1 }^{\tau_{\star}+  \kappa} \frac{2^{(1-\alpha+\epsilon)M_t}}{\min\{|X(\tau_{\star})|+|s-\tau_{\star}|,1\}}  d s\big| \lesssim 2^{(1-\alpha+2\epsilon)M_t }.
 \ee

 To sum up, in whichever case, from the estimates (\ref{march17eqn81}), (\ref{march17eqn62}), (\ref{march17eqn63}),  and (\ref{march17eqn64}), the following estimate holds for any $t_1, t_2\in  [T^{\ast},  T^{\ast\ast} ]\subset[0,t]$,
 \be
   \big||V(t_1)| -   |V(t_2)| \big|\lesssim2^{(1-\alpha)M_t+2\epsilon M_t}< 2^{(1-2\beta)M_t},\ee
   \be
    \Longrightarrow \frac{3}{4} 2^{(1-\beta) M_t-1}  \leq \inf_{s \in [T^{\ast}, T^{\ast}+\epsilon]}   |V(s)|\leq \sup_{s \in [T^{\ast}, T^{\ast}+\epsilon]}  |V(s)| \leq   \frac{5}{4}2^{(1-\beta)M_t-1}  
 \ee
Hence improving the bootstrap assumption in (\ref{march12eqn31}). Therefore, we can extend the size of $\epsilon$ such that $T^{\ast\ast} =t$.  To sum up, in whichever case, we have $\sup_{s   \in [0, t]}   |V(s )| <     2^{(1-\beta)M_t  } .  $ Hence finishing the proof of (\ref{jan13eqn11}).

\end{proof}

 \noindent \textit{Proof of Theorem} \ref{maintheorem}  :\qquad Recall (\ref{jan18eqn1}).   
From the conservation law (\ref{conservationlaw}), we have 
\be\label{march20eqn11}
 \big| \int_{\R^3}\int_{|v|\leq 2^{(1-\beta+\epsilon)M_t }}(1+|v|)^{N_1} f(t,x,v)   d x dv\big|\lesssim  2^{(N_1-1)(1-\beta+\epsilon)M_t}\leq (\tilde{M}_{N_1}(t))^{(1-\beta+\epsilon)}. 
\ee

Recall the definition of the majority set $R(t,s)$ in (\ref{march19eqn31}). From the relation (\ref{jan13eqn11}) in Proposition \ref{mainprop}, we have  $|X(0;t,x ,v ) |+|  V(0;t,x ,v ) |\gtrsim 2^{\beta M_t }$ if    $|x|\gtrsim 2^{\beta M_t  } $ or $|v|\gtrsim 2^{(1-\beta)M_t}$. Therefore, the polynomial decay of  the initial data in (\ref{assumptiononinitialdata}) implies that the following estimate holds if $|v|\gtrsim 2^{(1-\beta)M_t}$, 
\[
|f(t,x,v)|=| f_0(X(0;t,x ,v ), V(0;t,x ,v ))| \lesssim (1+|X(0;t,x ,v )|+|V(0;t,x ,v )|)^{-N_0}\lesssim 2^{-10N_1 M_t/3}.
\]
Moreover, if $|x|\geq 2^{2\epsilon M_t}$, then from the estimate (\ref{march17eqn11}), we have
\[
|X(t,0,x,v)|\geq |x| - t\geq |x|- 2^{\epsilon M_t}\gtrsim (1+|x|),
\]
If  $|v|\geq  2^{3(1+4\epsilon)M_t}$, then the following estimate holds from the equation (\ref{jan12eqn21}) and the first  estimate  in (\ref{feb25eqn72}) in Lemma \ref{roughalongchar2}, 
\[
|V(t,0,x,v)|\geq |v|- \int_{0}^t \big( |E(s,X(s)| + |B(s,X(s)|\big)d s \gtrsim (1+ |v|). 
\]
Therefore, from the above three estimates and the assumption of initial data in (\ref{assumptiononinitialdata}),   the following estimate holds if  $|v|\gtrsim 2^{(1-\beta)M_t}$   regardless the size of $|x|$, 
\be\label{feb1eqn2}
|f(t,x,v)|= |f_0(X(t,0,x,v),V(t,0,x,v))|\lesssim (1+|x|)^{-4}(1+|v |)^{- N_1-4}.  
\ee
Therefore, from the above estimate, we have
\be\label{march20eqn12}
 \big| \int_{\R^3}\int_{|v|\geq 2^{(1-\beta+\epsilon)M_t}}(1+|v|)^{N_1} f(t,x,v)   d x dv\big| \lesssim 1.
\ee
To sum up, from the estimates (\ref{march20eqn11}) and (\ref{march20eqn12}), we have
\[
M_{N_1}(t)\lesssim \big( \tilde{M}_{N_1}(t) \big)^{  1-\beta +\epsilon }. 
\]
Since the above estimate holds for any $t\in [0, T)$ and $\tilde{M}_n(t)$ is an increasing function with respect to $t $, the following estimate holds for any $s\in [0, t]$, 
\[
M_{N_1}(s)\lesssim \big( \tilde{M}_{N_1}(s) \big)^{   1-\beta  +\epsilon}\leq \big( \tilde{M}_{N_1}(  t) \big)^{  1-\beta+\epsilon },
\]
Hence
\be\label{march21eqn21}
\tilde{M}_{N_1}(t)=\sup_{s\in [0, t]}M_{N_1}(s)+  (1+t)^{N_1^2}  \lesssim \big( \tilde{M}_{N_1}(  t) \big)^{  1-\beta+\epsilon }+  (1+t)^{ N_1^2}, 
 \quad \Longrightarrow \tilde{M}_{N_1}(t)\lesssim   (1+t)^{N_1^2}. 
\ee
 Therefore, from the above estimate and  the first   estimate  in (\ref{feb25eqn72}) in Lemma \ref{roughalongchar2}, we have
\be\label{march21eqn22}
\sup_{x_0, v_0\in \R^3}\sup_{t\in[0,T)} \sum_{j\in \Z_{+}}\int_{0}^{T}    |K_j(s, X(s;t,x_0,v_0 ))|  ds  \lesssim   (1+T)^{4{N_1} }. 
 \ee
  Hence finishing the proof of our desired estimate (\ref{march13eqn21}) and the theorem from the the decomposition of the electromagnetic field in (\ref{march14eqn1}). 
 \qed
 
\section{Rough estimates of the electromagnetic field}\label{roughcontrol}

In this section, we provide several rough estimates for the  electromagnetic field in terms of the moment of the distribution function, which will be used as basic tools in the next section for more sophisticated analysis of the increment of the sizes of characteristics. In particular, as stated in (\ref{march3eqn51}) and (\ref{feb25eqn72}), we give a point-wise estimate for the electromagnetic field and give an upper bound for the targeted quantity in (\ref{march13eqn21}) respectively.

To prove our desired rough estimates, for simplicity,  we don't distinguish the frequencies of the electromagnetic field. The Glassey-Strauss decomposition as stated in the following Lemma  is very convenient and  useful.

\begin{lemma}\label{glasseystraussdec}
For any $j\in \Z_{+}$, the following decomposition holds 
\be\label{gsdecomposition}
K_j(t,x)=   \sum_{-j\leq l\leq 2, l\in \mathbb{Z}}  \int_0^t   K_{S;j  }^l(t,s,x) + K_{T;j }^l (t,s,x) d s, \quad K\in \{E, B\}, 
\ee
where     
\be\label{feb23eqn3}
E_{T;j}^{l}(t,s,x):=   \int_{\R^3} \int_{\mathbb{S}^2} \frac{(1-|\hat{v}|^2)(\hat{v}+\theta)  }{(1+\hat{v}\cdot \theta)^2} f(s, x+(t-s)\theta, v)   \varphi_j(v)\varphi_{l;-j}(\angle(v, -\theta))  d \theta d v   ,
\ee
\be\label{feb23eqn4}
B_{T;j,l}^{l}  (t,s,x):=     \int_{\R^3} \int_{\mathbb{S}^2}   \frac{(1-|\hat{v}|^2) (\theta\times \hat{v})  }{(1+\hat{v}\cdot \theta)^2} f(s, x+(t-s)\theta, v)   \varphi_j(v)\varphi_{l;-j}(\angle(v, -\theta)) d \theta d v  ,
\ee
\be\label{feb23eqn5}
E_{S;j }^{l}(t,s,x):=  \int_{\R^3} \int_{\mathbb{S}^2} (t-s) K(s, x+(t-s)\theta, v)\cdot  \nabla_v\big(\frac{ \varphi_j(v)(\hat{v}+\theta) }{1+\hat{v}\cdot \theta}\big)  \varphi_{l;-j}(\angle(v, -\theta))     f(s, x+(t-s)\theta,v) d \theta d v  , 
\ee
\be\label{feb23eqn6}
B_{S;j }^{l}(t,s,x) =   \int_{\R^3} \int_{\mathbb{S}^2} (t-s)   K(s, x+(t-s)\theta, v) \cdot \nabla_v\big( \frac{ \varphi_j(v) (\theta\times \hat{v})  }{(1+\hat{v}\cdot \theta) } \big)\varphi_{l;-j}(\angle(v, -\theta))     f(s, x+(t-s)\theta,v) d \theta d v . 
\ee

\end{lemma}
\begin{proof}
The desired decomposition (\ref{gsdecomposition}) follows from redoing the Glassey-Strauss decomposition  first for $K_j, K\in\{E,B\},$ and then used a dyadic decomposition with threshold $-j$ for the angular between $v$ and $-\theta$.  See \cite{glassey2}[Theorem 3] for more details. 
\end{proof}

\begin{lemma}\label{roughestimateofelectromag}
For any $x\in \R^3/\{0\},s\in [0, t]$, the following estimate holds  
\be\label{march3eqn51}
\sum_{j\in \mathbb{Z}_+}\sum_{K\in\{E,B\}}|K_j(s,x)|  \lesssim \frac{2^{(1+\epsilon) M_t} }{r_{-}}, \quad r_{-}:=\min\{1,r\}, \quad  \epsilon:=60/{N_1}=6\times 10^{-4}. 
\ee
Moreover, we have the following estimate as a byproduct, 	
\be\label{march4eqn1}
 \sum_{j\in \mathbb{Z}_+,  j\notin[(1-\alpha)M_t, (1+\epsilon) M_t]    }\sum_{K\in \{E, B\}}  K_j(s)\lesssim \frac{ 2^{(1-\alpha+\epsilon)M_t}}{r_{-}}, \quad \alpha:=1/100=3\beta.
\ee
\end{lemma}
\begin{proof}
 Recall the decomposition of electromagnetic field in (\ref{march14eqn1}) and the Glassey-Strauss decomposition in (\ref{gsdecomposition}). 

 We first estimate the ``T'' part. Recall (\ref{feb23eqn3}) and (\ref{feb23eqn4}).  Note that, from the estimate (\ref{feb29eqn1}) in Lemma \ref{basicestimateint}, the estimate (\ref{feb25eqn12}), and the volume of support of $v$ and $\theta$, we have
\be\label{feb29eqn2}
| E_{T;j }^{l} (s,\tau,x) | + | B_{T;j }^{l} (s,\tau,x) | \lesssim 2^{-j-2l} \min\{\frac{1}{r|s-\tau||r-(s-\tau)|}   \min\{2^{-j}, 2^{- {N_1} j+ {N_1} M_t}  \}, 2^{3j+2l}\}.
\ee

Now we estimate the ``S'' part. Recall (\ref{feb23eqn5}) and (\ref{feb23eqn6}). From the Cauchy-Schwarz inequality and  the estimates in Lemma \ref{basicestimateint}, the estimate (\ref{feb25eqn12}), and the volume of support of $v$, we have
\[
| E_{S;j }^{l} (s,\tau,x) | + | B_{S;j }^{l} (s,\tau,x) | \lesssim 2^{-j-2l}    |s-\tau|\big(\int_{\R^3} \int_{\mathbb{S}^2}  f (\tau, x+(s-\tau)\theta, v)  \varphi_j(v) \varphi_{l; -j}(\angle(v, -\theta)) d v d \theta  \big)^{1/2} \]
\[
\times \big(\int_{\R^3} \int_{\mathbb{S}^2} (|E(\tau, x+(s-\tau)\theta)|^2 + |B(\tau, x+(s-\tau)\theta)|^2) \varphi_j(v) \varphi_{l; -j}(\angle(v, -\theta)) d v d \theta  \big)^{1/2}
\]
\be\label{feb29eqn3}
\lesssim 2^{-j-2l}  |s-\tau|\big(  \frac{2^{3j +2l}}{r|s-\tau||r-(s-\tau)|} \big)^{1/2} \big[     \min\{ \frac{  \min\{2^{-j},  2^{- {N_1} j+ {N_1} M_t}   \}}{r|s-\tau||r-(s-\tau)|}, 2^{3j+2l}\}\big]^{1/2}. 
\ee
Based on the possible size of $j$, we separate into two cases as follow.

\noindent \textbf{Case $1$:}\quad If $ j\geq (1+\epsilon/2)M_t.$ \quad From the estimate (\ref{feb29eqn2}) and the estimate (\ref{feb29eqn3}), we have
\[
\sum_{j\in \mathbb{Z}_{+},  j\geq (1+\epsilon/2)M_t} \sum_{l\in[-j,2]\cap \mathbb{Z}} \sum_{K\in\{E,B\}} \int_0^s |K_{T;j }^{ l}(s,\tau,x)|  + | K_{S;j }^{ l} (s,\tau,x) |  d \tau
\]
\[
\lesssim  \sum_{j\in \mathbb{Z}_{+},  j\geq (1+\epsilon/2)M_t}  (1+j)\big[ \int_{0}^{s} \big(2^{2j}\big)^{9/10} \big( \frac{2^j 2^{-{N_1} j + {N_1} M_t} }{r|s-\tau||r-(s-\tau)|}\big)^{1/10} d\tau \]
\be\label{march3eqn53}
 + \int_{0}^{s} (s-\tau)^{1/4} 2^{3j/2}\big( \frac{1}{r |r-(s-\tau)|}\big)^{1/2} \big( \frac{ 2^{-  {N_1} j +  {N_1} M_t}}{r |r-(s-\tau)|}\big)^{1/4}\big( 2^{3j+2l}\big)^{1/4} d s \lesssim \frac{1}{r_{-}^{3/4}}.
\ee

\noindent \textbf{Case $2$:}\quad If $j\leq (1+\epsilon/2)M_t.$\qquad Let $\delta:=  2^{-(1+\epsilon/2)M_t} $. From the estimate (\ref{feb29eqn2}), we have
\[
 \sum_{j\in \mathbb{Z}_{+}, j\leq (1+\epsilon/2)M_t } \sum_{l\in[-j,2]\cap \mathbb{Z}} \sum_{K\in\{E,B\}}\int_0^s |K_{T;j }^{ l}(s,\tau,x)| d \tau \lesssim  \sum_{j\in \mathbb{Z}_{+}, j\leq (1+\epsilon/2)M_t } \sum_{l\in[-j,2]\cap \mathbb{Z}}\int_{[s-\delta,s]\cup [s-r-\delta,s-r+\delta] } 2^{2j} d \tau \]
\be\label{march4eqn2}
 + \int_{[0,s-\delta]\cap [s-r-\delta,s-r+\delta]^c} \frac{2^{-2j-2l}}{r|s-\tau||r-(s-\tau)|}  d \tau
\lesssim \frac{ (M_t)^2 }{r \delta} +  M_t 2^{(2+\epsilon)M_t}\delta \lesssim \frac{2^{(1+\epsilon)M_t}}{r_{-}}.
\ee
Let $\tilde{\delta}:= 2^{-nM_t}$. From the estimate (\ref{feb29eqn3}), we have 
\[
  \sum_{j\in \mathbb{Z}_{+}, j\leq (1+\epsilon/2)M_t }\sum_{l\in[-j,2]\cap \mathbb{Z}}  \sum_{K\in\{E,B\}} \int_0^s |K_{S;j  }^{ l}(s,\tau,x)|   d \tau\]
 \[
 \lesssim   \sum_{j\in \mathbb{Z}_{+}, j\leq (1+\epsilon/2)M_t } \sum_{l\in[-j,2]\cap \mathbb{Z}}\int_{[0,s]\cap [s-r-\tilde{\delta},s-r+\tilde{\delta}]^c}  \frac{ 2^{-l}}{r|r-(s-\tau)|} d\tau  \]
\be\label{march3eqn52}
  +  \int_{[0,s]\cap [s-r-\tilde{\delta},s-r+\tilde{\delta}]} \frac{ (1+j) 2^{2j} (s-\tau)^{1/2}}{(r|r-(s-\tau)|)^{1/2}} d\tau\lesssim \frac{  2^{(1+\epsilon/2)M_t} M_t }{r}  + \frac{ 1}{r^{1/2}} \lesssim \frac{  2^{(1+\epsilon)M_t}}{r_{-}}. 
\ee
  To sum up, our desire estimate (\ref{march3eqn51}) holds from the estimates (\ref{march3eqn53}), (\ref{march4eqn2}),  and (\ref{march3eqn52}).

  By using the same strategy used in above estimates and letting  $\delta:= (  \tilde{M}_n(t) )^{-(1 -\epsilon/2)/(n-1)} $ and   $\tilde{\delta}:= (  \tilde{M}_n(t) )^{-1}$, we obtain our desired estimates (\ref{march4eqn1})   after combining the estimate (\ref{march3eqn53}) and  changing the range of the summation with respect to $j$ in (\ref{march3eqn52}) from 
$ j\in \mathbb{Z}_{+},  {j}\leq(1+\epsilon/2)M_t $ to $ j\in \mathbb{Z}_{+},  j\leq (1-\alpha)M_t$. 

\end{proof}

\begin{lemma}\label{roughestimateelectro}
The following estimate holds for any $s\in [0,t], j\in \Z_{+},x\in \R^3,$ s.t.,   $ j\leq (1+\epsilon)M_t$, 
 \be\label{march15eqn35} 
|E_j(s,x)| +|B_{j}(s,x)|\lesssim  2^{3(1+2\epsilon)M_t}.
\ee
\end{lemma}
 \begin{proof}
 Note that our desired estimate (\ref{march15eqn35})
is a trivial consequence of (\ref{march3eqn51}) in Lemma \ref{roughestimateofelectromag} if $|x|\geq 1$. It would be sufficient to consider the case $|x|\leq 1$.  Recall  the Glassey-Strauss decomposition in (\ref{gsdecomposition}). From the obtained estimate (\ref{feb29eqn2}), we have
 \be\label{march15eqn34} 
 \sum_{l\in[-j,2]\cap \Z} \int_{0}^s | E_{T;j }^{l} (s,\tau,x) | + | B_{T;j }^{l} (s,\tau,x) | d \tau\lesssim  (1+j) |t| 2^{2j}  \lesssim 2^{2(1+2\epsilon)M_t}.
 \ee

 Hence, it remains to estimate the ``$S$''   part. Recall (\ref{feb23eqn5}) and (\ref{feb23eqn6}).  From the rough estimate of the electromagnetic field (\ref{march3eqn51}) in Lemma \ref{roughestimateofelectromag} and  the obtained estimate (\ref{feb29eqn3}), we know that the following estimate holds if $\tau \neq s-r$, 
 \be\label{march15eqn31} 
 \sum_{K\in\{E, B\}} |K_{S;j}^l(s,\tau,x)| \lesssim \min\{\frac{(s-\tau)2^{(1+\epsilon)M_t}}{|r-(s-\tau)|} 2^{-j-2l} 2^{3j+2l},2^{2(1+\epsilon)M_t} \big(\frac{s-\tau}{r|r-(s-\tau)|}\big)^{1/2}\}.
 \ee 
From the estimate  (\ref{march15eqn31})  , we have
 \[
  \sum_{K\in\{E, B\}} |K_{S;j}(s,x)| \lesssim \int_{s-r- r/2}^{s-r+ r/2}2^{2(1+\epsilon)M_t} \big(\frac{s-\tau}{r|r-(s-\tau)|}\big)^{1/2} d\tau + \int_{[0,s]\cap [s-r-r/2, s-r+r/2]^c } \frac{ (s-\tau)2^{3(1+\epsilon)M}}{|r-(s-\tau)|}  d\tau
 \]
 \be\label{march15eqn33} 
 \lesssim 2^{2(1+\epsilon)M_t} r^{1/2} + 2^{3(1+ \epsilon)M_t}\big(1+|t\ln(t)|+ r\ln(1/r)\big)\lesssim 2^{3(1+2\epsilon)M_t}. 
 \ee
 To sum up, our desired estimate (\ref{march15eqn35}) holds from the estimates (\ref{march15eqn34}) and (\ref{march15eqn33}).  
 \end{proof}

 \begin{lemma}\label{roughalongchar2}
 For any $t\in [0, T)$, the following estimates hold  for any initial data  ${x}_0,  {v}_0\in \R^3$ of characteristics,  
\be\label{feb25eqn72} 
\sum_{j\in \Z_{+}}\sum_{K\in \{E, B\}}\int_{0}^{t}       |K_j(s, X(s ))|  ds \lesssim   2^{3(1+3\epsilon)M_t}, \quad 
\sum_{j\in \Z, j\geq (1+\epsilon)M_{t}} \sum_{K\in \{E,B\}}\int_{0}^{t}  |K_j(s, X(s ))| ds \lesssim   1.
 \ee
 \end{lemma}
 \begin{proof}
  Recall the decomposition of electromagnetic field in (\ref{march14eqn1}) and the Glassey-Strauss decomposition in (\ref{gsdecomposition}). Note that, the following estimate holds from the estimate (\ref{march3eqn51}) in Lemma \ref{roughestimateofelectromag}, the estimate (\ref{march3eqn53}), the estimate (\ref{march15eqn35}) in Lemma\ref{roughestimateelectro}, 
  \be\label{march20eqn35}
  \sum_{K\in\{E, B\}} |K(t,x)| \lesssim \big( \frac{1}{r^{3/4}} +2^{3(1+2\epsilon)M_t}  \big)\psi_{\leq 2}(r) + 2^{(1+2\epsilon)M_t } \psi_{\geq 2}(r). 
  \ee
Moreover,  the following rough  estimates hold  straightforwardly, 
\[
 \sum_{K\in \{E, B\}}  \int_{0}^{t} \int_{0}^s      |K_{T;j }^{ l}(s,\tau, X(s ))|    d\tau ds\]
 \be\label{march13eqn31}
  \lesssim     2^{j}\int_{0}^{t} \int_{\tau}^{t}  \int_{\R^3} \int_{0}^{2\pi} \int_0^\pi       f( \tau,  X(s )+(s-\tau) \omega,v)  \varphi_j(v)\varphi_{l;-j}(\angle(v, - \omega))  \sin \theta d\theta d \phi  d v d s   d \tau,
\ee
\[
 \sum_{K\in \{E, B\}}  \int_{0}^{t} \int_{0}^s      |K_{S;j }^{ l}(s,\tau, X(s ))|    d\tau ds
  \lesssim  \sum_{K\in \{E, B\}}      2^{ j }  \int_{0}^{t} \int_{\tau}^{t}  \int_{\R^3} \int_{0}^{2\pi} \int_0^\pi  |K(\tau,  X(s )+(s-\tau) \omega) | 
\]
\be\label{feb25eqn81}
  \times  \varphi_j(v)\varphi_{l;-j}(\angle(v, - \omega))  (s-\tau)  f( \tau,  X(s )+(s-\tau) \omega,v)    \sin \theta d\theta d \phi  d v d s   d \tau, 
\ee  
 As observed by Pallard \cite{pallard2}, for fixed $\tau$, we can do change of coordinates $(s, \theta,\phi)\longrightarrow (X(s )+(s-\tau)\omega)$. As a result of direct computation (see also \cite{pallard2}), the Jacobian of the transformation is $(X'(s)\cdot \omega+1)(s-\tau)^2 \sin  \theta $.
From the H\"older inequality, we have
 \[
\big|  \int_{\tau}^{t}  \int_{0}^{2\pi} \int_{0}^{\pi}h(\tau, X(s ) \omega ) \sin \theta d  \theta  d\phi d s \big| 
 \lesssim  \big| \int_{\tau}^{t}  \int_{0}^{2\pi}  \int_{0}^{\pi}|h(\tau, X(s )+(s-\tau) \omega )|^p (X'(s)\cdot \omega+1)(s-\tau)^2  \sin \theta d  \theta  d\phi d s \big|^{1/p}
 \]
 \be\label{feb25eqn87}
\times \big| \int_{\tau}^{t}  \int_{0}^{2\pi}  \int_{0}^{\pi}\big((X'(s)\cdot \omega+1)(s-\tau)^2\big)^{-q/p} (\sin\theta)^{(1-1/p)q} d  \theta  d\phi d s\big|^{1/q}, \quad \frac{1}{p} + \frac{1}{q}=1. 
 \ee
 Note that the following estimate holds for any $q \in(1, 3/2)$, 
 \[
  \int_{\tau}^{t}  \int_{0}^{2\pi}  \int_{0}^{\pi}\big((X'(s)\cdot \omega+1)(s-\tau)^2\big)^{-q+1}  \sin\theta  d  \theta  d\phi d s \lesssim (1+t)^2. 
 \]
 Let $  p=3/(1-1/n) $  and $q=13$.   From the estimates (\ref{march20eqn35}--\ref{feb25eqn87})  and the estimate (\ref{feb25eqn12}), we have
 \[
  \sum_{j\geq (1+\epsilon)M_{t}} \sum_{K\in \{E, B\}}  \sum_{-j\leq l\leq 2, l\in \mathbb{Z}} \int_{0}^{t} \int_{0}^s     |K_{S;j }^{ l}(s,\tau, X(s ))|  d\tau ds +  \int_{0}^{t} \int_{0}^s     |K_{T;j }^{ l}(s,\tau, X(s ))|  d\tau ds 
  \]
    \be\label{march20eqn31}
 \lesssim   \sum_{j\geq (1+\epsilon)M_{t}} \sum_{K\in \{E, B\}}   2^{ j }  (1+t)^2\big[t \int_0^t  \| \int_{\R^3} |K(\tau, x)| f(\tau, x,v )  \varphi_j(v)  d v\|_{L^p_x}    d\tau 
 + \int_0^t   \| \int_{\R^3}   f(\tau, x,v )  \varphi_j(v)  d v\|_{L^p_x}    d\tau \big]
 \ee
\[
\lesssim    \sum_{j\geq (1+\epsilon)M_{t}} \sup_{\tau \in[0,t]} 2^{ j }  (1+t)^4\big[ 2^{3(1+2\epsilon)M_t}   \| \int_{\R^3}   f(\tau, x,v )  \varphi_j(v)  d v\|_{L^p_x} + \|  \int_{\R^3} r^{-3/4} \psi_{\leq 2}(r)   f(\tau, x,v )  \varphi_j(v)  d v\|_{L^p_x}\big]
\]
\be\label{march20eqn38}
\lesssim   \sum_{j\geq (1+\epsilon)M_{t}}  2^j \big[ 2^{3(1+3\epsilon)M_t}  2^{3j(1-1/p)} \big(    2^{- {N_1} j + {N_1} M_t}    \big)^{1/p} +  2^{3j(1-1/q)} \big(   2^{- {N_1} j + {N_1} M_t}    \big)^{1/q}\big]\lesssim 1. 
\ee
Our desired estimates in (\ref{feb25eqn72}) hold from the above estimate and the estimate (\ref{march15eqn35}) in Lemma (\ref{roughestimateelectro}).  
 \end{proof}

\section{Increment of the sizes of characteristics over time}\label{incrementestimate}

Recall the equations satisfied by the magnitudes of characteristics in (\ref{jan12eqn21}). As stated in the following Proposition, our main goal in this section is to quantitatively control the increment of the size of characteristics over time.

 \begin{proposition}\label{mainproposition}
Under the assumptions in \textup{(\ref{aproiriestimate})} and \textup{(\ref{march12eqn31})}, the following estimate holds for any $t_1, t_2\in [T^{\ast}, T^{\ast\ast} ]\subset[0, t]$, 
\[
\big| \int_{t_1}^{t_2} \tilde{V}(s)  \cdot E(s, X(s  ) ) d s \big| \lesssim   \int_{t_1}^{t_2}  |K(s, X(s))|  \min\{  \frac{2^{- 3 M_t/2+24\alpha M_t}}{|X(s)|_{-} },2^{ M_t/2+3\alpha  M_t}|X(s)|+ 2^{-2M_t /3+8\alpha M_t} \}  ds
\]
\be\label{march14eqn41}
 +  2^{M_t/3  +7\alpha M_t}+  \int_{t_1}^{t_2}   \frac{ 2^{(1-\alpha+\epsilon) M_t}}{|X(s)|_{-}}  d s  , \quad \alpha:=1/100, \quad \epsilon:=60/N_1=6\times 10^{-4},
\ee
\[
\big| \int_{t_1}^{t_2}C_1(X(s), V(s)) \cdot E(s, X(s  ) ) d s \big| + \big|\int_{t_1}^{t_2} C_2(X(s), V(s)) \cdot  B(s,X(s  ) )  d s\big| \lesssim 2^{-(\alpha-2\beta-3\epsilon)  M_t} |t_2-t_1|
\]
\be\label{march16eqn51}
 +   \sum_{i=1,2}|X(t_i)| \min\{ \frac{2^{-2M_t+15\alpha M_t  }}{|X(t_i)|_{-}}, 2^{M_t/2+3\alpha M_t } |X(t_i)| + 2^{ -2M_t/3+ 8\alpha  M_t  }\}, 
\ee
where $C_1(\cdot, \cdot)$ and $C_2(\cdot, \cdot)$ are defined in \textup{(\ref{march17eqn1})}. 
 \end{proposition}
 \begin{proof}
 To improve presentation, we  postpone the proof of the above two estimates   to the end of this subsection. 
 \end{proof}

Recall the decomposition of electromagnetic field in (\ref{march14eqn1}). Because  the rough estimate of $K_j, K\in \{E, B\}$, in (\ref{march4eqn1}) in Lemma \ref{roughestimateofelectromag} is  sufficient for our purpose, we only need  to consider the uncovered case  $ j\in[ (1-\alpha)M_t , (1+\epsilon)M_t  ]  $. In the rest of this section, we use a Fourier method to prove our desired estimates for the uncovered case.

 First of all, we   introduce the set-up of different frameworks for the electromagnetic field in different scenarios.  Recall the detailed formulas of $E_j(t,x)$ and $B_j(t,x)$ in (\ref{march14eqn4}) and (\ref{march14eqn5}).   After represented them on Fourier side and using the Vlasov equation to substitute $\p_t f$, the following decomposition holds after doing dyadic decompositions for both the frequency and the angle between $\xi$ and $v$, 
\be\label{march16eqn41}
K_j(t,x) = \sum_{k\in \mathbb{Z}} \sum_{l\in [-j,2]\cap \Z} \frac{-i}{2}\big( K_{k;j,l}(t,x) - \overline{ K_{k;j,l}(t,x)}\big), \quad K\in \{E, B\},
\ee
where
\be\label{jan14eqn51}
B_{k;j,l}(t,x) = \int_0^t V_{k;j,l}^b(t,s,x) d s,\quad E_{k;j,l}(t,x) = \int_0^t V_{k;j,l}^{e}(t,s,x) + High_{k;j,l}^{0}(t,s,x) d s,
\ee
\be\label{jan14eqn53}
 V_{k;j,l}^u(t,s,x) =  \int_{\R^3} \int_{\R^3} e^{i x\cdot \xi + i t |\xi| -is |\xi| } m_u(\xi,v)    \widehat{f}(s, \xi, v)  \varphi_{k;j,l}(\xi, v )d \xi  d v  d s, \quad u\in \{e,b\}
\ee
\be\label{feb28eqn11}
 High_{k;j,l}^{0}(t,s,x) := \int_{\R^3} \int_{\R^3} \int_{\R^3} e^{i x\cdot \xi + i t |\xi| -is |\xi| } \frac{1}{  |\xi|} \nabla_v \hat{v}\cdot \big(\widehat{E}(s, \xi-\eta) + \hat{v}\times \widehat{B}(s, \xi-\eta) \big)   \widehat{f}(s, \eta, v) \varphi_{k;j,l}(\xi, v ) d \xi  d\eta  d v  , 
\ee
where the cutoff function $\varphi_{k;j,l}(\xi, v )$ and the symbols $m_u(\xi, v), u\in\{e,b\},$ are defined as follows, 
\be\label{feb28eqn12}
\varphi_{k;j,l}(\xi, v ):=\psi_k(\xi) \varphi_j(v) \varphi_{l;-j}(\angle(\xi, -v)), \quad m_e(\xi, v):=  i \frac{(\hat{v}\cdot \xi)\hat{v}-\xi}{|\xi|}   , \quad m_b(\xi, v):=  -i\frac{\hat{v}\times \xi}{|\xi|}. 
\ee
where the cutoff functions $\varphi_{j}(\cdot)$ and $\varphi_{l;-j}(\cdot)$ are defined in (\ref{cutoffwith}).

   We remind readers that the angular cutoff function in (\ref{feb28eqn12}), which measure the angle between frequency $\xi$ and $v$,  different from the one we used in the Glassey-Strauss decomposition (\ref{gsdecomposition}), which measure the angle in physical space. 

Recall (\ref{mainequation}). To take account of the linear effect, we study the profile $g(t,x,v):=f(t,x+t\hat{v}, v)$ of the distribution function $f(t,x,v)$ as follows, 
\be\label{march19eqn1}
\p_t g(t,x,v)= (E(t, x+t\hat{v})+\hat{v}\times B(t,x+t\hat{v}))\cdot \nabla_v f(t,x+t\hat{v}, v), \quad f(t,x,v)=g(t,x-t\hat{v},v).	
\ee

 To exploit the oscillation with respect to time in some cases, we  will  use the normal form transformation for $ V_{k;j,l}^u(t,s,x)$ in (\ref{jan14eqn53}).  More precisely, after doing integration by parts with respect to $s$,  the following equality holds,
 \[
 \int_0^t V_{k;j,l}^u(t,s,x) d s= \int_0^t \int_{\R^3} \int_{\R^3} e^{i x\cdot \xi + i t |\xi| -is |\xi| - is  \hat{v}\cdot\xi } m_u(\xi,v)    \widehat{g}(s, \xi, v)  \varphi_{k;j,l}(\xi, v )d \xi  d v  d s
 \] 
\be\label{feb28eqn1}
 = End_{k;j,l}^u(t,t,x)-End_{k;j,l}^u(t,0,x) + \int_0^t High^u_{k;j,l}(t,s,x) d s , \quad u\in\{e,b\},
\ee
where
\[
 End_{k;j,l}^u(t,t, x)=   \int_{\R^3} \int_{\R^3} e^{i x\cdot \xi - i t\hat{v}\cdot \xi} \frac{ m_u(\xi, v) \varphi_{k;j,l}(\xi, v )}{|\xi|(|\xi|+\hat{v}\cdot \xi )}   \widehat{g}(t, \xi, v) d \xi  d v  
\]
\be\label{jan17eqn5}
  =  \int_{\R^3} \int_{\R^3} e^{i x\cdot \xi  }  \frac{ m_u(\xi, v) \varphi_{k;j,l}(\xi, v )}{|\xi|(|\xi|+\hat{v}\cdot \xi )}    \widehat{f}(t, \xi, v)  d \xi  d v,
\ee
\be\label{march19eqn2}
 End_{k;j,l}^u(t,0, x)=   \int_{\R^3} \int_{\R^3} e^{i x\cdot \xi +it|\xi|} \frac{ m_u(\xi, v) \varphi_{k;j,l}(\xi, v )}{|\xi|(|\xi|+\hat{v}\cdot \xi )}   \widehat{f}(0, \xi, v) d \xi  d v  ,
\ee
\be\label{jan17eqn6}
 High_{k;j,l}^u(t,s, x):= \int_{\R^3} \int_{\R^3} e^{i x\cdot \xi + i t |\xi| -is |\xi|- i s\hat{v}\cdot \xi} \frac{ m_u(\xi, v)}{|\xi|(|\xi|+\hat{v}\cdot \xi)}   \p_t \widehat{g}(s, \xi, v)  \varphi_{k;j,l}(\xi, v ) d \xi  d v
\ee
\be\label{feb28eqn59}
= \int_{\R^3} \int_{\R^3} \int_{\R^3} e^{i x\cdot \xi + i t |\xi| -is |\xi| } \nabla_v\big( \frac{  m_u(\xi, v)  \varphi_{k;j,l}(\xi, v ) }{|\xi|(|\xi|+\hat{v}\cdot \xi)}    \big)\cdot  \big(\widehat{E}(s, \xi-\eta)+\hat{v}\times \widehat{B}(s, \xi-\eta) \big) \hat{f}(s, \eta, v) d \eta d \xi d v.  
\ee
Since $End_{k;j,l}^u(t,0, x)$ only depends on the initial data, which is regular, it is uniformly bounded over time.

For   simplicity of notation, we formulate $High_{k;j,l}^m(t,s,x), m\in\{0,e,b\},$ in (\ref{feb28eqn11}) and  (\ref{feb28eqn59}),  uniformly as follows, 
\be\label{march4eqn76}
High_{k;j,l}^m(t,s,x)=  \int_{\R^3} \int_{\R^3} \int_{\R^3} e^{i x\cdot \xi + i t |\xi| -is |\xi| } \tilde{m}_{k,j,l;m} ( \xi, v )   \big(\widehat{E}(s, \xi-\eta) +\hat{v}\times \widehat{B}(s, \xi-\eta) \big) \hat{f}(s, \eta, v) d\eta d v d \xi 
\ee 
where
\be\label{feb28eqn74}
 \tilde{m}_{k,j,l;0} ( \xi, v )=\varphi_{k;j,l}(\xi, v )  (  |\xi|)^{-1}\nabla_v \hat{v}, \quad \quad  \tilde{m}_{k,j,l;u} ( \xi, v )=  \nabla_v\big( \frac{ m_u(\xi,v) \varphi_{k;j,l}(\xi, v ) }{|\xi|(|\xi|+\hat{v}\cdot \xi)}    \big),\quad u\in\{e,b\}.
 \ee 

As pointed out by Klainerman-Staffilani \cite{Klainerman3}(see also the smoothing effect pointed out by Bouchut-Golse-Pallard \cite{bouchut}), a very important observation, which will also be used here, is that the   integration of electromagnetic field along the characteristic is smoother than the  electromagnetic field itself. Intuitively speaking, the smoothing effect comes from the different speeds of wave and particles.

 However,   we need to be very careful about the gain and the loss of using the smoothing effect, which, in practice, is  doing integration by parts with respect to ``$s$''. This process will be carried out  in (\ref{feb28eqn68}). 

As a matter of fact, the loss of doing integration over time along the spatial characteristic $X(s)$ depends on    the angle between the frequency and the velocity characteristics $V(s)$, see the resulted symbol in (\ref{feb28eqn72}).  Presumably, the loss  can be very large  because another electromagnetic field, which is the main unknown,  will be introduced  when ``$\p_s$'' hits the velocity characteristics, see (\ref{feb28eqn71}) and the system of equations satisfied by characteristics in (\ref{characteristicseqn}).

  To better measure the loss, we dyadically localize the   angle between the frequency variable and the velocity characteristics $V(s)$, which is fixed at any fixed time $s$. Motivated from the above discussion and the general form of nonlinearities in (\ref{march4eqn76}),  we  study    the following bilinear form, 
\be\label{march10eqn1}
T_{k;j}^{ l,n}(K, f)(s,\tau,x):= \int_{\R^3} \int_{\R^3} \int_{\R^3} e^{i x\cdot \xi + i (s-\tau) |\xi|   }\varphi_{k;j,l}(\xi, v )  m_1 ( \xi, v ) m_2 ( \xi, \omega_0 ) \varphi_{n;l}(\angle(\xi, \omega_0))   \widehat{K}(\tau, \xi-\eta)  \hat{f}( \tau, \eta, v) d\eta d v d \xi,
\ee
where  $\omega_0\in \mathbb{S}^2$ is fixed, $k\in \mathbb{Z}, j\in\Z_{+}$, $l\in [-j,2]\cap \Z$,   $n\in [l, 2]\cap \mathbb{Z}$ ,  $K\in \{E, B\}$, the cutoff function $\varphi_{k;j,l}(\xi)$ was defined  in (\ref{feb28eqn12}), and the symbols $m_1(\xi, v)$ and $m_2(\xi, \omega_0)$ satisfies the following assumption for any $k\in \mathbb{Z}, j\in \mathbb{Z}_{+}$,
\be\label{symbolassumption1}
\sup_{   |a|\leq   N_1^2 }  \sup_{\xi, v\in \R^3}  \big[2^{ak}\big|\big(\tilde{v}\cdot \nabla_\xi\big)^a m_1 ( \xi, v )\big| + 2^{a(k+l)}\big|\big(\tilde{v}\times  \nabla_\xi\big)^a m_1 ( \xi, v )\big|\big]\varphi_{k;j,l}(\xi, v )\lesssim \| m_1\|_{\mathcal{S}^\infty_{k,j,l}},
\ee
\be\label{symbolassumption2}
\sup_{   |a|\leq    N_1^2 } \sup_{\xi \in \R^3, \omega_0\in \mathbb{S}^2}  \big[2^{ak}\big|\big(\omega_0\cdot \nabla_\xi\big)^a m_2( \xi, \omega_0 )\big| + 2^{a(k+n)}\big|\big( \omega_0\times  \nabla_\xi\big)^a m_2 ( \xi, \omega_0 )\big|\big] \psi_{[k-4,k+4]}(\xi) \varphi_{n;l}(\angle(\xi, \omega_0)) \lesssim \| m_2\|_{\mathcal{S}^\infty_{k;n,l }}.
\ee
 We use  $\| m_1\|_{\mathcal{S}^\infty_{k,j,l}}$ and $\| m_2\|_{\mathcal{S}^\infty_{k }}$, which are not norms,  simply for notational convenience.

In   applications, $\omega_0$ will be the direction (module the sign) of velocity characteristics $\tilde{V}(s)$,   the symbol  $m_1 ( \xi, v ) $ will be the price of using the normal form transformation or one (if not using the normal form transformation), and the symbol $m_2 ( \xi, \omega_0 ) $ will be the price of  doing integration over time along the spatial characteristic $X(s)$ or one. More precisely, recall (\ref{march4eqn76}), after doing dyadic decomposition for the angle between $\xi$ and $-V(s)$, we have
\[
High_{k;j,l}^m(s,\tau,x)= \sum_{n\in [l,2]
\cap \Z} H_{k;j,l}^{m;n}(s,\tau,x),\quad  H_{k;j,l}^{m;n}(s,\tau,x):= \int_{\R^3} \int_{\R^3} \int_{\R^3} e^{i x\cdot \xi + i (s-\tau) |\xi|   } \tilde{m}_{k,j,l;m} ( \xi, v )
\]
\be\label{march12eqn10}
 \times \varphi_{n;l}(\angle(\xi, -\tilde{V}(s)))    \big(\widehat{E}(\tau, \xi-\eta) +\hat{v}\times \widehat{B}(\tau, \xi-\eta) \big) \hat{f}(\tau, \eta, v) d\eta d v d \xi. 
\ee
 
Let   $C:\R_x^3 \times R_v^3\longrightarrow \R_x^{3 }$  be a differentiable function.  To control the increment of magnitude of characteristics between any two time $t_1, t_2\in [T^{\ast},T^{\ast\ast} ]\subset[0,t]$,  we estimate the following general  form, which is applicable for   quantities in (\ref{jan12eqn21}) and (\ref{march1eqn2100}).  Moreover, recall (\ref{march12eqn10}),  the following equality holds after doing integration by parts with respect to ``$s$'',
\be\label{feb28eqn68}
\int_{t_1}^{t_2} \int_0^s C( X(s ), V(s ))  \cdot  H_{k;j,l}^{m;n}(s ,\tau,X(s)) d \tau  d s    = \sum_{i=1,2,3}G_{ m}^{n;i}(C)(t_1,t_2),   
 \ee  
where $m\in \{0,e,b \}$ and the detailed formulas of $G_{ m}^{n;i}(C)(t_1,t_2),  i\in\{1,2,3\},$ are given as follow, 
\be\label{feb28eqn70}
G_{ m}^{n;1}(C) (t_1,t_2)=-  \int_{t_1}^{t_2}  \int_{\R^3}  \int_{\R^3}  \int_{\R^3}   e^{i X(s  ) \cdot \xi  }     \mathcal{M}_{ m}(C)(s, \xi, v ) \cdot \big[\big(\widehat{E}(s ,\xi-\eta) + \hat{v}\times \widehat{B}(s, \xi-\eta) \big) \widehat{f}(s,\eta, v) \big]  d v d \xi d\eta  d s,
 \ee
\[ 
G_{ m}^{n;2}(C) (t_1,t_2)= \sum_{i=1,2}(-1)^i \int_0^{t_i}  \int_{\R^3}  \int_{\R^3} \int_{\R^3}    e^{i X(t_i ) \cdot \xi + i  t_i |\xi| -i \tau |\xi| }   \mathcal{M}_{ m}^n(C)(t_i,\xi, v )\cdot \big[\big(\widehat{E}(\tau, \xi-\eta)
\]
\be\label{feb28eqn69}
   + \hat{v}\times \widehat{B}(\tau, \xi-\eta) \big)  \widehat{f}(\tau,\eta, v) \big]
  d v d \tau d \xi d \eta,
 \ee
 \[
G_{ m}^{n;3} (C)(t_1,t_2)= -  \int_{t_1}^{t_2} \int_0^s  \int_{\R^3}  \int_{\R^3} \int_{\R^3} \big( e^{i X(s ) \cdot \xi + i s |\xi| -i \tau |\xi| }\big)   \p_s     \mathcal{M}_{ m}^n(C)(s, \xi, v )\cdot   \big[\big(\widehat{E}(\tau, \xi-\eta) \]
\be\label{feb28eqn71}
+ \hat{v}\times \widehat{B}(\tau, \xi-\eta) \big) \widehat{f}(\tau,\eta, v) \big]    d v d \tau  d s   d \xi,
\ee
where  
\be\label{feb28eqn72}
  \mathcal{M}_{ m}^n(C)(s, \xi, v ):=   \big( \  \tilde{m}_{k,j,l;m} ( \xi, v )\big)^{\top}	C( X(s ), V(s ))  \big({ (|\xi|+\hat{V}(s  )\cdot \xi  ) } \big)^{-1}  \varphi_{n;l}(\angle(\xi,- \tilde{V}(s ))) . 
\ee

Recall that $j\in [(1-\alpha)M_t, (1+\epsilon)M_t]$. For the sake of readers,  based  on the possible size of $k$ (corresponds to the size of output frequency $\xi$),    $l$ (corresponds to the size of angle between $\xi$ and $-v$) and  $n$(corresponds to the size of angle between $\xi$ and velocity characteristic $-V(s)$), we summarize our  main strategies as follows.
\begin{enumerate}
\item[$\bullet$] The case $k\leq 40M_t$ and $k\leq  -l +j + (1- 3\alpha/2)M_t$. \quad For this case, we work with the formulation (\ref{jan14eqn51}) directly. We don't use the normal form transformation and don't use the integration by parts in characteristic time. 
\item[$\bullet$] The case $k\leq 40 M_t$ and $k\geq -l +j + (1-3\alpha/2)M_t$.\quad For this case, we  always use the normal form transformation and work with the formulation  (\ref{feb28eqn1}). Moreover, we don't use the integration by parts in characteristic time if $l$ is large or $n$ is small and do use it if   $l$ is small and $n$ is large. 

\item[$\bullet$] The    case $k\geq 40 M_t$. \quad For the very large frequency case, we use both the normal form transformation and the integration by parts in characteristic time. 
\end{enumerate}
 \begin{lemma}\label{nonormalform}
For any fixed $k\in \mathbb{Z}, j\in\mathbb{Z}_{+}$, $l\in [-j, 2]\cap \mathbb{Z}$, s.t., $k\leq 40M_t, j\in [(1-\alpha)M_t, (1+\epsilon)M_t]$, and $k\leq   -l +j + (1-3\alpha/2)M_t$,   the following estimate holds for any $s\in [0, t]$, 
\be\label{march2eqn40}
\sum_{K\in \{E, B\}}\big|K_{k;j,l}(s,x) \big|  \lesssim \min\{ \frac{2^{  (1-\alpha)M}  }{r_{-}},  2^{3k+3l-j }+ 2^{2k+2l  } \} , \quad r:=|x|.
\ee
 
\end{lemma}
\begin{proof}
Recall the decomposition (\ref{jan14eqn51}) and the detailed formulas of $V_{k;j,l}^u(t,s,x), u\in\{e,b\}$, and $High_{k;j,l}^{0}(t,s,x)$ in  (\ref{jan14eqn53}--\ref{feb28eqn11}). From the volume of support of $\xi$ and the estimate (\ref{feb25eqn12}), we have
\[
\sum_{u\in \{e,b\}}| V_{k;j,l}^u(s,\tau,x)| \lesssim 2^{3k+3l} \|\widehat{f}(\tau, \xi, v)\varphi_j(v)\|_{L^1_v L^\infty_\xi} \lesssim 2^{3k+3l-j},  
\]
\be 
 |High_{k;j,l}^{0}(s,\tau,x)| \lesssim  \sum_{K\in \{E, B\} } 2^{2k+2l-j} \| |K(\tau,x) f(\tau,x,v) \varphi_j(v)\|_{L^1_{x,v}}\lesssim  2^{2k+2l}. 
\ee
From the above two estimates, we know that the second part of our desired estimate (\ref{march2eqn40}) holds.  Moreover, it allows us to focus on the case  $k\geq (1-\alpha)M_t/2$ for obtaining   first part of the desired estimate in (\ref{march2eqn40}).

From the Kirchhoff's formulas (\ref{dec20eqn1}) and (\ref{march4eqn100}) in Lemma \ref{Kirchoff}, we have
\be\label{feb27eqn14}
V_{k;j,l}^u(s,\tau,x)   =  I_u^1(s,\tau,x)  + I_u^2(s,\tau,x) ,\quad I_u^1(s,\tau,x):=  \int_{\mathbb{S}^2} \int_{\R^3} \int_{\R^3} (s-\tau) K_{k,l}^{u}(y,\theta, v) f(\tau,x-y-(s-\tau)\theta,v) \psi_j(v) d y  d v d \theta 
\ee
\be
 I_u^2(s,\tau,x):=  \int_{\mathbb{S}^2} \int_{\R^3} \int_{\R^3} E_{k,l}^{u}(y , v) f(\tau,x-y-(s-\tau)\theta,v) \psi_j(v) d y  d v d \theta ,\quad u\in \{e,b\},
\ee
\[
 High_{k;j,l}^{0}(s,\tau,x)  =  J_u^1(s,\tau,x)  + J_u^2(s,\tau,x) , \quad  J_u^1(t,s,x):=  \int_{\mathbb{S}^2} \int_{\R^3} \int_{\R^3} (s-\tau) K_{k,l}^{high;0}(y,\theta, v)  \big(E(\tau,x-y-(s-\tau)\theta )
\]
\be\label{marcheqn3}
  + \hat{v}\times B(\tau,x-y-(s-\tau)\theta) \big) \cdot \nabla_v \hat{v} f(\tau,x-y-(s-\tau)\theta,v)  \psi_j(v) d y  d v d \theta  ,
\ee
\[
 J_u^2(s,\tau,x):=  \int_{\mathbb{S}^2} \int_{\R^3} \int_{\R^3}   E_{k,l}^{high;0}(y,v)  \big(E(\tau,x-y-(s-\tau)\theta )  + \hat{v}\times B(\tau,x-y-(s-\tau) \theta) \big) \cdot \nabla_v \hat{v}
\]
\be\label{march11eqn110}
 \times f(\tau,x-y- (s-\tau)\theta,v)  \psi_j(v) d y  d v d \theta  ,
\ee
where
\be 
  K_{k,l}^u(y,\theta, v):=\int_{\R^3} e^{i y\cdot \xi} i  m_u(\xi, v) (|\xi|+\theta\cdot\xi) \varphi_{l;-j}(\angle(\xi, -v))\psi_k(\xi) d \xi,\quad u\in\{e,b\},
  \ee
   \be
   K_{k,l}^{high;0}(y,\theta, v):=\int_{\R^3} e^{i y\cdot \xi} (1+\theta\cdot \xi/|\xi|) \varphi_{l;-j}(\angle(\xi, -v))\psi_k(\xi) d \xi.
\ee
  \be 
  E_{k,l}^u(y,v):=\int_{\R^3} e^{i y\cdot \xi}    m_u(\xi, v)   \varphi_{l;-j}(\angle(\xi, -v))\psi_k(\xi) d \xi,\quad u\in\{e,b\}. 
  \ee
  \be
    E_{k,l}^{high;0}(y,v)= \int_{\R^3} e^{i y\cdot \xi}|\xi|^{-1} \varphi_{l;-j}(\angle(\xi, -v))\psi_k(\xi) d \xi
  \ee
  After doing integration by parts in $\tilde{v}$ direction and  $\tilde{v}^{\bot}$ directions, the following estimate holds for any $u\in\{e,b\},$
  \[
     \sup_{\theta\in \mathbb{S}^2} 2^{-k-l}| K_{k,l}^u(y,\theta, v) | +  |K_{k,l}^{high;0}(y,\theta, v)|+   2^{-l}|E_{k,l}^u(y,v) |   + 2^k| E_{k,l}^{high;0}(y,v)|
  \]
\be\label{march11eqn91}
  \lesssim 2^{3k+ 2l} (1+2^k|y\cdot \tilde{v}|)^{-10}(1+2^{k+l}|y\times \tilde{v}|)^{-10}. 
\ee

 Note that, for $y\in \R^3$ in (\ref{feb27eqn14}--\ref{march11eqn110}), we can  represent it in terms of the cylinder coordinate system  as follows, 
 \be\label{marcheqn16eqn69}
y= \big(\tilde{v}\cdot y\big) \tilde{v} + r_y \omega_y,  \quad r_y\in \R_{+}, \omega_y\in S_{v}, \quad S_v:=\{ \omega\in \S^2: \omega\cdot v =0\}.
 \ee
Therefore, from the estimate of kernel in (\ref{march11eqn91}), the estimate (\ref{feb29eqn1}) in Lemma \ref{basicestimateint}, and the estimate (\ref{feb25eqn12}),  we have
\[
\big| I_u^1(s,\tau,x)  \big|  \lesssim     \int_{\mathbb{S}^2}   \int_{\R^3}  \int_{ S_v} \int_{\R} \int_{\R}  (s-\tau)   2^{4k+3l} (1+a2^k)^{-10}(1+r 2^{k+l})^{-10} f(\tau,x-a\tilde{v}-r\omega_v-(s-\tau)\theta,v) \psi_j(v)  rdr d a d \omega_{v}   d v d \theta 
\]
\be\label{march16eqn61}
\lesssim \frac{2^{k+l } \| \int_{\R^3} f(\tau, x,v)\varphi_j(v) d v  \|_{L^1_x} }{r|r-(t-s)|} \lesssim    \frac{2^{k+l-j}}{r|r-(t-s)|}.
\ee
Following the same strategy, we have 
\be\label{march2eqn34}
  \big| I_u^2(s,\tau,x)  \big|    
\lesssim  \frac{2^{ l-j}}{r(t-s)|r-(t-s)|}. 
\ee
\be
|  J_u^1(s,\tau,x) | \lesssim   \sum_{K\in\{E,B\}} \frac{2^{-j}}{r|r-(t-s)|}    \|\int_{\R^3} |K(\tau,x)|f(\tau,x,v) \varphi_j(v)  d v      \|_{L^1_x}\lesssim \frac{1}{r|r-(t-s)|}  
\ee
\be\label{march11eqn131}
|  J_u^2(s,\tau,x) | \lesssim   \sum_{K\in\{E,B\}} \frac{2^{-k-j}}{r(t-s)|r-(t-s)|}    \|\int_{\R^3} |K(\tau,x)|f(\tau,x,v) \varphi_j(v)  d v      \|_{L^1_x}\lesssim \frac{2^{-k }}{r(t-s)|r-(t-s)|}. 
\ee
Moreover, from the estimate of kernels in (\ref{march11eqn91}), Cauchy-Schwarz inequality  and the volume of support of $v$, we have
\be\label{march11eqn93}
\big| I_u^1(s,\tau,x)  \big|    
\lesssim 2^{k+l+3j+\epsilon M_t},\quad \big| I_u^2(s,\tau,x)  \big|     
\lesssim 2^{l+3j}, \quad \big| J_u^1(s,\tau,x)  \big|  \lesssim 2^{3k/2+l+2j+\epsilon M_t}, \big| J_u^2 (s,\tau,x)  \big|  \lesssim 2^{ k/2+l+2j}. 
\ee
Let $\delta= 2^{-10M_t}$.  From the above estimate (\ref{march11eqn93}) and the estimates (\ref{march16eqn61}--\ref{march11eqn131}), the following estimate holds for any $u\in\{e,b\}$,
\[
 \int_0^s \big( \big| I_u^1(s,\tau ,x)  \big|   +\big| J_u^1(s,\tau,x)  \big|\big) d \tau \]
 \[
  \lesssim \int_{[0,s]\cap [s-r-\delta,  s-r+\delta] }2^{\epsilon M_t}( 2^{k+l+3j} +2^{3k/2+l+2j}) d\tau +\int_{[0,s]\cap [s-r-\delta,  s-r+\delta]^c }      \frac{2^{k+l-j} +1}{r|r-(s-\tau)|} d \tau
\]
\be\label{march11eqn139}
\lesssim  \delta2^{\epsilon M_t}  ( 2^{k+l+3j} +2^{3k/2+l+2j})       + \ln(1/\delta)\frac{1}{r} (2^{k+l-j} +1)\lesssim \frac{2^{ (1-\alpha) M_t} }{r_{-}}, 
\ee
\[
 \int_0^s\big( \big| I_u^2(s,\tau ,x)  \big| +  \big| J_u^2(s,\tau ,x)  \big|\big)  d \tau \]
 \[
  \lesssim    \int_{ [0,s- {\delta}]\cap [s-r-\delta,  s-r+\delta]^c }    \frac{2^{ -j} +2^{-k}}{r(s-\tau)|r-(s-\tau)|} d\tau+    \int_{([0,s- {\delta}]\cap [s-r-\delta,  s -r+\delta])\cup[s- {\delta}, t] }   2^{3j}  d\tau \]
\be\label{march11eqn100}
\lesssim   {\delta} 2^{3 j} + \frac{(2^{-j}+2^{-k})\ln(1/ {\delta})}{  r^2}\lesssim \frac{2^{ -(1-\alpha-\epsilon)M_t/2 }}{r^2_{-}}.
\ee
Note that  the above estimate (\ref{march11eqn100}) is sufficient for the estimate of $I_u^2(s,\tau ,x)$ and $J_u^2(s,\tau ,x)$ if $r\geq 2^{ -3(1-\alpha-\epsilon)M_t/2 }$.  If $r\leq 2^{ -3(1-\alpha-\epsilon)M_t/2 }$, then from the estimates (\ref{march2eqn34}), (\ref{march11eqn131}), and (\ref{march11eqn93}),   the following estimate holds,
\[
\int_0^s \big( \big| I_u^2(s,\tau ,x)  \big| +  \big| J_u^2(s,\tau ,x)  \big|\big)  d \tau  \lesssim \int_{ s- 2^{ - M_t  }}^{s} (2^{l+3j} + 2^{k/2+l+2j}) d\tau + \int_0^{\max\{s- 2^{ - M_t  },0\} }       \frac{2^{ -j} +2^{-k}}{r(s-\tau)|r-(s-\tau)|}  d\tau
\]
\be\label{march19eqn7}
\lesssim 2^{2(1+2\epsilon)M_t} + \frac{2^{M_t}(2^{ -j} +2^{-k})}{r} \lesssim \frac{2^{(1-\alpha)M_t}}{r}
\ee

Recall the decomposition (\ref{jan14eqn51}) and  the decompositions in (\ref{feb27eqn14}) and  (\ref{marcheqn3}). To sum up, our desired estimate (\ref{march2eqn40}) holds from the estimates (\ref{march11eqn139}), (\ref{march11eqn100}), and (\ref{march19eqn7}). 

\end{proof}

Now, we proceeds to consider the case when $k\geq -l +j + (1-3\alpha/2)M_t $. As we explained at the beginning, we will use normal form transformation for this case. Recall (\ref{feb28eqn1}). The estimate of endpoint case is summarized in the following Lemma. 
\begin{lemma}\label{endpointcase}
For any fixed $k\in \mathbb{Z}, j\in\mathbb{Z}_{+}$, $l\in [-j, 2]\cap \mathbb{Z}$, s.t.,  $j\in[(1-\alpha)M_t, (1+\epsilon)M_t],$ $k\geq -l +j + (1-3\alpha/2)M_t $,     the following estimate holds for any $s\in[0,t]$,
\be\label{march2eqn53}
|End_{k;j,l}^e(s,s,x)| + |End_{k;j,l}^b(s,s,x)| \lesssim  \min\{ \frac{2^{(1-\alpha)M_t}}{r_{-}} , 2^{-k-l + 3j }\}. 
\ee
\end{lemma}
\begin{proof}
Recall the detailed formula of $ End_{k;j,l}^u(s,s,x), u\in\{e,b\}$,  in (\ref{jan17eqn5}). In terms of kernel, we have
\[
End_{k;j,l}^u(s,s,x)= \int_{\R^3} \int_{\R^3} \widetilde{K_{k,l}^u}(y, v) f(s,x-y, v) \varphi_j(v) d y d  v, \quad u\in \{e,b\},
\] 
where
\[
 \widetilde{K_{k,l}^u}(y, v)= \int_{\R^3} e^{i y \cdot \xi}  \frac{ m_u(\xi, v)\psi_k(\xi) }{|\xi|(|\xi|+\hat{v}\cdot \xi )}  \varphi_{l;-j}(\angle(v, -\xi)) d \xi   .
\]
Recall the detailed formulas of symbol $m_u(\xi, v)$ in (\ref{feb28eqn12}). After doing integration by parts in $\tilde{v}$ direction and  $\tilde{v}^{\bot}$ directions, the following estimate holds for the kernels,
\be\label{kernelnormal}
 |  \widetilde{K_{k,l}^u}(y, v)|  \lesssim 2^{2k+l} (1+2^k|y\cdot \tilde{v}|)^{-100/\epsilon}(1+2^{k+l}|y\times \tilde{v}|)^{-100/\epsilon}. 
\ee
From the above estimate of kernels and the volume of support of $v$, we have
\be\label{march2eqn49}
|End_{k;j,l}^u(s,s,x) \lesssim 2^{-k-l + 3j }\lesssim 2^{(1+3\alpha)M_t}
\ee
In particular, from the above estimate, we have the following estimate if $|x|\leq 2^{-4\alpha M_t}$, 
\be\label{march8eqn21}	
|End_{k;j,l}^e(s, s,x)| + |End_{k;j,l}^b(s,s,x)| \lesssim  {r}^{-1}  2^{(1-\alpha)M_t}.
\ee
Now, we restrict ourself to the case $|x|\geq 2^{-4\alpha M_t}$.  Due to the fast decay rate  of kernels in the estimate (\ref{kernelnormal}),  the following estimate holds if $|y| \geq   2^{-k-l+\epsilon M_t}$, 
\be\label{march2eqn47}
\sum_{u\in\{e,b\}} \big|\int_{\R^3} \int_{|y|\geq 2^{-k-l+\epsilon M_t} } \widetilde{K_{k,l}^u}(y, v) f(s,x-y, v) \varphi_j(v) d y d  v \big| \lesssim 2^{-k-l-50 j} \lesssim r^{-1}. 
\ee
Note that $|x-y|\sim |x|$ if $|y|\leq 2^{-k-l+\epsilon M_t} $ and $|x|\geq 2^{-4\alpha M_t}$. From the radial symmetry of $f(t,\cdot,\cdot)$ and the estimate of kernel in (\ref{kernelnormal}), the following estimate holds if   $|x|\geq 2^{-4\alpha M_t}$,
\be\label{march2eqn48}
\sum_{u\in\{e,b\}} \big|\int_{\R^3} \int_{|y|\leq 2^{-k-l+\epsilon M_t} } \widetilde{K_{k,l}^u}(y, v) f(s,x-y, v) \varphi_j(v) d y d  v \big| \lesssim 2^{2k+l} \frac{2^{-2k-2l+2\epsilon  M_t }}{r^2} 2^{-j}\lesssim \frac{2^{2\epsilon M_t}}{r^2} \lesssim \frac{2^{5\alpha M_t}}{r}
\ee
 To sum up, our desired estimate (\ref{march2eqn53}) holds from the estimates (\ref{march2eqn49}), (\ref{march8eqn21}), (\ref{march2eqn47}), and (\ref{march2eqn48}).
\end{proof}

Recall the decompositions in  (\ref{jan14eqn51}) and    (\ref{feb28eqn1}).  Now our mission is reduced to estimate the quadratic terms, $High^m_{k;j,l}(t,s,x),$ $ m\in\{0,e,b\} $ in (\ref{march4eqn76}) for the case $k\geq -l +j + (1-3\alpha/2)M_t $.

Recall the decomposition of $High^m_{k;j,l}(t,s,x)$,$ m\in\{0,e,b\} $, in    (\ref{march12eqn10}). The key ingredients of estimating $H_{k;j,l}^{m;n}(t,s,x)$, are two point-wise bilinear estimates in Lemma \ref{bilinearestimaterough}, which are applicable regardless whether using the   integration by parts in characteristic time. We will   use them as   black boxes for the estimate of  $H_{k;j,l}^{m;n}(t,s,x)$, $ m\in\{0,e,b\} $, in the next two Lemmas. For the sake of clarity, we postpone and  and  elaborate the proof of Lemma  \ref{bilinearestimaterough} to the next subsection.

\begin{lemma}\label{largelandsmalln}
Under the assumptions \textup{(\ref{aproiriestimate})} and \textup{(\ref{march12eqn31})},  for any fixed $s  \in [T^{\ast}, T^{\ast\ast} ]\subset[0,t]$,  $k\in \mathbb{Z}, j\in\mathbb{Z}_{+}$, $l\in [-j, 2]\cap \mathbb{Z}$, s.t.,  $k\leq 40M_t, j\in [(1-\alpha)M_t,(1+\epsilon)M_t]$,  and $k\geq  -l +j + (1-3\alpha/2)M_t $,   if moreover we have $l\geq -2(1-7\alpha)M_t/3$ or $n\leq -M_t/2-7\alpha M_t$, then  the following estimate holds,
\be\label{march16eqn3}
 \big| \int_0^s C_1(X(s),V(s))\cdot \big( H_{k;j,l}^{0;n}(s,\tau,X(s )) +  H_{k;j,l}^{e;n}(s,\tau,X(s))\big)  d \tau \big| + \big| \int_0^s  C_2(X(s),V(s))\cdot  H_{k;j,l}^{b;n}(s,\tau ,X( s))  d \tau \big|\lesssim 2^{-\alpha M_t },
\ee
\be\label{march4eqn92}
  \big| \int_0^s \tilde{V}(s)\cdot \big( H_{k;j,l}^{0;n}(s,\tau  ,X(s)) +  H_{k;j,l}^{e;n}( s,\tau ,X(s))\big)  d  \tau \big|  \lesssim      \min\{ \frac{ 2^{(1-\alpha) M_t}}{|X(s)|_{-}}, 2^{k/2+3(1+3\epsilon)M_t }|X(s)|  + 2^{3(1+3\epsilon)M_t }\} . 
\ee
\end{lemma}
\begin{proof}
Recall the detailed formulas of symbols $ \tilde{m}_{k,j,l;m} ( \xi, v ) $ in (\ref{feb28eqn74}) and smooth coefficients in (\ref{march17eqn1}).  Note that, the following decomposition holds for any $u$, 
\be\label{march14eqn45}
u\cdot \nabla_v = u\cdot \tilde{v}(\tilde{v}\cdot \nabla_v) + \sum_{i=1,2,3} u\cdot(\tilde{v} \times \mathbf{e}_i)\big( (\tilde{v} \times \mathbf{e}_i)\cdot \nabla_v  \big),
\ee 
where $\mathbf{e}_1:=(1,0,0), \mathbf{e}_2:=(0,1,0), \mathbf{e}_3:=(0,0,1).$ Recall (\ref{symbolassumption1}). With the above equality, as a result of direct computations, we have
\be\label{march14eqn80}
\|\tilde{m}_{k,j,l;0} ( \xi, v )\|_{\mathcal{S}^\infty_{k,j,l}}    \lesssim 2^{-k-j },\quad \|\tilde{m}_{k,j,l;e} ( \xi, v )\|_{\mathcal{S}^\infty_{k,j,l}}+  \|\tilde{m}_{k,j,l;b} ( \xi, v )\|_{\mathcal{S}^\infty_{k,j,l}}    \lesssim 2^{-k-j-2l }, 
\ee
\be\label{march15eqn1}
\| \tilde{V}(s)\cdot \tilde{m}_{k,j,l;e} ( \xi, v )  \varphi_{n;l}(\angle(\xi, -\tilde{V}(t)))  \|_{\mathcal{S}^\infty_{k,j,l}}\lesssim 2^{-k-j-2l +n},\ee
\[
   \|C_1(X(s),V(s))  \tilde{m}_{k,j,l;e} ( \xi, v )   \varphi_{n;l}(\angle(\xi, -\tilde{V}(t)))  \|_{\mathcal{S}^\infty_{k,j,l}}   + \|   C_2(X(s),V(s))  \tilde{m}_{k,j,l;b} ( \xi, v ) \varphi_{n;l}(\angle(\xi, -\tilde{V}(t)))  \|_{\mathcal{S}^\infty_{k,j,l}}
\]
\be\label{march14eqn51}
   \lesssim |X(s)| \big( |\tilde{X}(s)\times \tilde{V}(s)| + 2^{n}\big) 2^{-(1-\beta)M_t-k-j-2l }. 
\ee
Therefore, from the above estimates of symbols, the estimate (\ref{march17eqn11}),   and the estimates (\ref{march10eqn2}) and (\ref{march14eqn100}) in Lemma \ref{bilinearestimaterough}, we have
\[
 \big| \int_0^s C_1(X(s),V(s))\cdot \big( H_{k;j,l}^{0;n}(s,\tau,X(s )) +  H_{k;j,l}^{e;n}(s,\tau,X(s))\big)  d \tau \big| + \big| \int_0^s  C_2(X(s),V(s))\cdot  H_{k;j,l}^{b;n}(s,\tau ,X(ts))  d \tau \big|
\]
\[
 \lesssim |X(s)| 2^{-(1-\beta)M_t-k-j-2l} \frac{ 2^{k+l/2+ n+(1+3\alpha) M_t } }{|X(s)|_{-}}  \lesssim   2^{-\alpha M_t}.
\]
\[
\big| \int_0^s \tilde{V}(s)\cdot \big( H_{k;j,l}^{0;n}(s,\tau  ,X(s)) +  H_{k;j,l}^{e;n}( s,\tau ,X(s))\big)  d  \tau \big|  \lesssim \min\{\frac{2^{-3l/2+n+4\alpha M_t}}{|X(s)|_{-}},  2^{k/2+3(1+3\epsilon) M_t }|X(s)|  + 2^{3(1+3\epsilon)M_t }\}. 
\]
Hence finishing the proof of our desired estimates (\ref{march16eqn3}) and  (\ref{march4eqn92}). 
\end{proof}

\begin{lemma}\label{smallllargen}
Under the assumptions \textup{(\ref{aproiriestimate})} and \textup{(\ref{march12eqn31})},  for any fixed $t_1, t_2  \in [T^{\ast},  T^{\ast\ast} ]\subset[0,t]$,  $k\in \mathbb{Z}, j\in\mathbb{Z}_{+}$, $l\in [-j, 2]\cap \mathbb{Z}$, s.t.,  $k\leq 40M_t, j\in [(1-\alpha)M_t,(1+\epsilon)M_t]$,  and $k\geq   -l +j + (1-3\alpha/2)M_t $,  if moreover we have  $l\leq -2(1-7\alpha)M_t/3$ and $n\geq -M_t/2-7\alpha M_t$, then    the following estimate holds,
\[
 \big| \int_{t_1}^{t_2} \int_0^s C_1(X(s),V(s))\cdot \big( H_{k;j,l}^{0;n}(s,\tau,X(s )) +  H_{k;j,l}^{e;n}(s,\tau,X(s))\big)  d \tau d s \big| + \big| \int_{t_1}^{t_2} \int_0^s  C_2(X(s),V(s))\cdot  H_{k;j,l}^{b;n}(s,\tau ,X( s))  d \tau d s \big|
\]
\be\label{march14eqn31}
\lesssim \sum_{i=1,2}|X(t_i)| \min\{ \frac{2^{-2M_t+15\alpha  M_t  }}{|X(t_i)|_{-}}, 2^{M_t/2+3\alpha M_t } |X(t_i)| + 2^{ -2M_t/3+8\alpha  M_t  }\} +  2^{-M_t/6+16\alpha  M_t}(t_2-t_1). 
\ee 
\[
\big| \int_{t_1}^{t_2} \int_0^s \tilde{V}\cdot \big( H_{k;j,l}^{0;n}(s,\tau,X(s )) +  H_{k;j,l}^{e ;n}(s,\tau,X(s))\big)  d \tau d s \big|\lesssim   2^{ M_t  /3+7\alpha M_t } + \int_{t_1}^{t_2} \min\{ \frac{ 2^{-M_t/2+10\alpha M_t }}{|X(s)|_{-}}, 2^{3M_t/2+4\alpha M_t}\} ds \]
\be\label{march16eqn11}
+   \int_{t_1}^{t_2}  |K(s, X(s))|  \min\{  \frac{2^{- 3M_t/2+24\alpha  M_t}}{|X(s)|_{-} },2^{ M_t/2+3\alpha M_t}|X(s)|+ 2^{-2M_t/3+8\alpha M_t} \}  ds. 
\ee

\end{lemma}
\begin{proof}
For this case, we do integration by parts in characteristic time. Recall the decomposition  (\ref{feb28eqn68}) and the associated symbol in (\ref{feb28eqn72}). Recall (\ref{symbolassumption2}). 
Based on possible destination of $\p_s  \mathcal{M}_{ m}^n(C)(s, \xi, v )$, we separate  $G_{ m}^{n;3}(C) (t_1,t_2)$ further into three parts as follows, 
\be\label{march15eqn83}
G_{ m}^{n;3}(C) (t_1,t_2)=G_{ m;1}^{n;3}(C) (t_1,t_2)+G_{ m;2}^{n;3}(C) (t_1,t_2)+  G_{ m;3}^{n;3}(C) (t_1,t_2)
\ee
\[
G_{ m;1}^{n;3}(C) (t_1,t_2) = -  \int_{t_1}^{t_2} \int_0^s  \int_{\R^3}  \int_{\R^3} \int_{\R^3} \big( e^{i X(s ) \cdot \xi + i s |\xi| -i \tau |\xi| }\big)         \hat{V}(s)\cdot \nabla_x C( X(s ), V(s ))  \tilde{m}_{k,j,l;m} ( \xi, v ) \cdot   \big[ \big(\widehat{E}(\tau, \xi-\eta)
 \]
\be\label{march14eqn59}
+ \hat{v}\times \widehat{B}(\tau, \xi-\eta) \big) \widehat{f}(\tau,\eta, v) \big] \big({ (|\xi|+\hat{V}(s  )\cdot \xi  ) } \big)^{-1}   \varphi_{n;l}(\angle(\xi, \tilde{V}(s)))  d v d \tau  d s   d \xi, 
\ee
\[
G_{ m;2}^{n;3}(C) (t_1,t_2) = -  \int_{t_1}^{t_2} \int_0^s  \int_{\R^3}  \int_{\R^3} \int_{\R^3} \big( e^{i X(s ) \cdot \xi + i s |\xi| -i \tau |\xi| }\big)        \widetilde{\mathcal{M}_{ m }^n}(C)(s, \xi  )\cdot   \big[ \  \tilde{m}_{k,j,l;m} ( \xi, v )\big(\widehat{E}(\tau, \xi-\eta) \]
\be\label{march14eqn121}
+ \hat{v}\times \widehat{B}(\tau, \xi-\eta) \big) \widehat{f}(\tau,\eta, v) \big]    \varphi_{n;l}(\angle(\xi, \tilde{V}(s)))  d v d \tau  d s   d \xi,
\ee
\[
G_{ m;3}^{n;3}(C) (t_1,t_2) = -  \int_{t_1}^{t_2} \int_0^s  \int_{\R^3}  \int_{\R^3} \int_{\R^3} \big( e^{i X(s ) \cdot \xi + i s |\xi| -i \tau |\xi| }\big)       C( X(s ), V(s )) \  \tilde{m}_{k,j,l;m} ( \xi, v ) \cdot   \big[\big(\widehat{E}(\tau, \xi-\eta) \]
\be\label{march14eqn64}
+ \hat{v}\times \widehat{B}(\tau, \xi-\eta) \big) \widehat{f}(\tau,\eta, v) \big] \p_s \big(  (|\xi|+\hat{V}(s  )\cdot \xi  )^{-1}     \varphi_{n;l}(\angle(\xi, \tilde{V}(s))) \big) d v d \tau  d s   d \xi,\quad a \in \{1,2\},
\ee
where
\be\label{march14eqn62}
 \widetilde{\mathcal{M}_{ m }^n}(C)(s, \xi  )= (K(s,X(s), V(s))\cdot \nabla_v C( X(s ), V(s ))   (|\xi|+\hat{V}(s  )\cdot \xi  )^{-1} 
.
\ee
Recall (\ref{symbolassumption2}) and (\ref{march17eqn1}). As a result of direct computations, we have
\be\label{march14eqn91}
\|   (|\xi|+\hat{V}(s  )\cdot \xi  )^{-1}\|_{\mathcal{S}^\infty_{k;n,l}} \lesssim 2^{-k-2n }, 
\ee
\be\label{march24eqn21}
 \|    \p_s \big(  (|\xi|+\hat{V}(s  )\cdot \xi  )^{-1}     \varphi_{n;l}(\angle(\xi, \tilde{V}(s))) \big)\|_{\mathcal{S}^\infty_{k;n,l}} \lesssim \sum_{K\in \{E, B\}}  |K(s, X(s))| 2^{-(1-\beta)M_t-k-3n},
\ee
\be\label{march14eqn57}
     \|\hat{V}(s)\cdot \nabla_x C_1( X(s ), V(s ))  \tilde{m}_{k,j,l;e} ( \xi, v ) \|_{\mathcal{S}^\infty_{k,j,l}} \lesssim 2^{-(1-\beta)M_t-k-j-2l+n}, \quad   \hat{V}(s)\cdot \nabla_x C_2( X(s ), V(s ))  =0, 
\ee
\be\label{march14eqn92}
\|   \widetilde{\mathcal{M}_{ m }^n}(C_1)(s, \xi  )\|_{\mathcal{S}^\infty_{k}}  + \|   \widetilde{\mathcal{M}_{ m }^n}(C_2)(s, \xi  )\|_{\mathcal{S}^\infty_{k;n,l}} \lesssim \sum_{K\in \{E, B\}} |X(s)||K(s, X(s))| 2^{-k-2n -2(1-\beta) M_t}. 
\ee 
\be
\|   \widetilde{\mathcal{M}_{ m }^n}(\tilde{V})(s, \xi  )\|_{\mathcal{S}^\infty_{k;n,l}}  \lesssim  \sum_{K\in \{E, B\}} |K(s, X(s))| 2^{-k-2n-(1-\beta) M_t}. 
\ee

With the above preparation, we are ready to estimate the terms in decompositions (\ref{feb28eqn68}) and (\ref{march15eqn83}) one by one.  Recall (\ref{feb28eqn70}). 
From the estimate (\ref{march4eqn19}) in Lemma \ref{firsttermestimate}, we have
\be\label{march16eqn1}
|G_{e}^{n;1}(C_1)(t_1,t_2)|+|G_{0}^{n;1}(C_1)(t_1,t_2)| + |G_{b}^{n;1}(C_2)(t_1,t_2)| \lesssim 2^{- M_t/2+10\alpha  M_t }(t_2-t_1), 
\ee
\be\label{march16eqn2}
|G_{ e}^{n;1}(\tilde{V}) (t_1,t_2)| + |G_{ 0}^{n;1}(\tilde{V}) (t_1,t_2)|\lesssim  \int_{t_1}^{t_2} \min\{ \frac{ 2^{ -M_t/2+10\alpha  M_t }}{|X(s)|_{-}}, 2^{3 M_t/2+4\alpha  M_t}\} ds.
\ee

Recall (\ref{feb28eqn69}). From the estimates (\ref{march14eqn51}) and  (\ref{march14eqn91}),    and the estimates (\ref{march10eqn2}) and (\ref{march14eqn100}) in Lemma \ref{bilinearestimaterough}, we have
\[
|G_{ e}^{n;2}(C_1) (t_1,t_2)| + |G_{ 0}^{n;2}(C_1) (t_1,t_2)| +  |G_{ b}^{n;2}(C_2) (t_1,t_2)|    \lesssim\sum_{i=1,2} 2^{-k-2n} 2^{-(1-\beta)M_t-k-j-2l} |X(t_i)|
\]
\[
 \times \min\{ \frac{2^{k+l/2+n+(1+ 3\alpha) M_t}}{|X(t_i)|_{-}}, 2^{3k/2+3j+2n +l+\epsilon  M_t}   |X(t_i)|  + 2^{k+3j+2n+2l+(1+ 3\epsilon) M_t}\}.
\]
\be\label{march16eqn8}
\lesssim \sum_{i=1,2}|X(t_i)| \min\{ \frac{2^{-2 M_t+15\alpha M_t  }}{|X(t_i)|_{-}}, 2^{M_t/2+2\alpha M_t } |X(t_i)| + 2^{ -2M_t/3+8\alpha M_t  }\}
\ee
 From the estimates (\ref{march15eqn1}) and (\ref{march14eqn91}) and the   estimates (\ref{march10eqn2}) and (\ref{march14eqn100}) in Lemma \ref{bilinearestimaterough}, we have
\[
|G_{ e}^{n;2}(\tilde{V}) (t_1,t_2)| + |G_{ 0}^{n;2}(\tilde{V}) (t_1,t_2)|\]
\[
\lesssim \sum_{i=1,2}  2^{-k-j-2l+n}2^{-k-2n} \min\{ \frac{2^{k+l/2+(1+3\alpha) M_t }}{|X(t_i)|_{-}},2^{3k/2+3j+2n +l+\epsilon  M_t } |X(t_i)| + 2^{k+3j+2n+2l+(1+3\epsilon) M_t }  \}
\]
\be\label{march16eqn14}
\lesssim \sum_{i=1,2} \min\{\frac{2^{- M_t  + 15\alpha  M_t   }}{|X(t_i)|_{-}}, 2^{3 M_t /2+2\alpha  M_t }|X(t_i)| + 2^{ M_t  /3+7\alpha M_t }\}\lesssim  2^{ M_t  /3+7\alpha M_t }. 
\ee
 
Recall (\ref{march14eqn59}). From the estimates (\ref{march14eqn57}) and (\ref{march14eqn80}) and the   estimates (\ref{march10eqn2}) and (\ref{march14eqn100}) in Lemma \ref{bilinearestimaterough}, we have
\[ 
|G_{ e;1}^{n;3}(C_1) (t_1,t_2)| + |G_{ 0;1}^{n;3}(C_1) (t_1,t_2)| +  |G_{ b;1}^{n;3}(C_2) (t_1,t_2)|\]
\[
 \lesssim \int_{t_1}^{t_2}2^{-(1-\beta)M_t-k-j-2l+n}2^{-k-2n}  \min\{\frac{2^{k+l/2+(1+3\alpha)M_t}}{|X(s)|_{-}},
 2^{3k/2+3j+2n +l+\epsilon  M_t}  |X(s)|  + 2^{k+3j+2n+2l+(1+3\epsilon) M_t} \} ds 
\]
\be\label{march16eqn9}
\lesssim \int_{t_1}^{t_2}\min\{\frac{2^{-2   M_t+16\alpha M_t }}{|X(s)|_{-}},2^{  M_t/2+3\alpha  M_t } |X(s)|+2^{-2 M_t/3+8\alpha   M_t} \} d s \lesssim  2^{-2 M_t/3+8\alpha   M_t}(t_2-t_1). 
\ee
Since the function $\tilde{V}:(X(s), V(s))\longrightarrow \tilde{V}(s)$ doesn't depend on $X(s)$,  we have
\be\label{march15eqn61}
|G_{ e;1}^{n;2}(\tilde{V}) (t_1,t_2)| + |G_{ 0;1}^{n;2}(\tilde{V}) (t_1,t_2)|=0.
\ee

Recall 	(\ref{march14eqn121}) and (\ref{march14eqn64}). From the estimates of symbols in  (\ref{march14eqn80}), (\ref{march14eqn51}), (\ref{march14eqn91}), (\ref{march24eqn21}), and (\ref{march14eqn92}),   the   estimates (\ref{march10eqn2}) and (\ref{march14eqn100}) in Lemma \ref{bilinearestimaterough}, and the rough estimate of the electromagnetic field (\ref{march3eqn51})   in Lemma \ref{roughestimateofelectromag}, we have
 \[ 
\sum_{a=2,3}|G_{ e;a}^{n;3}(C_1) (t_1,t_2)| + |G_{ 0;a }^{n;3}(C_1) (t_1,t_2)| +  |G_{ b;a}^{n;3}(C_2) (t_1,t_2)|\]
\[
 \lesssim \int_{t_1}^{t_2} |X(s)||K(s, X(s))| \min\{  2^{-k-2(1-\beta)M_t-3n} 2^{-k-j-2l}   \frac{2^{k+l/2+n+(1+3\alpha) M_t }}{|X(s)|_{-}}, \]
\[
  2^{-k-2(1-\beta) M_t -3n} 2^{-k-j-2l} \big[ 2^{3k/2+3j+2n +l+\epsilon  M_t}  |X(s)| + 2^{k+3j+2n+2l+(1+3\epsilon) M_t}\big]\} d s  
 \]	
 \be\label{march16eqn10}
 \lesssim \int_{t_1}^{t_2}  \min\{2^{(1 +20\alpha) M_t} |X(s)| +2^{-M/6+16\alpha M}, \frac{2^{- 3M/2+24\alpha M}}{|X(s)|_{-}}\}  ds  \lesssim 2^{-M/6+16\alpha M}(t_2-t_1). 
 \ee
From the estimates of symbols in  (\ref{march15eqn1}),  (\ref{march14eqn91}), and (\ref{march24eqn21}),  and the   estimates (\ref{march10eqn2}) and (\ref{march14eqn100}) in Lemma \ref{bilinearestimaterough}, we have 
  \[ 
\sum_{a=2,3}|G_{ e;a}^{n;3}(\tilde{V}) (t_1,t_2)| + |G_{ 0;a }^{n;3}(\tilde{V}) (t_1,t_2)| \lesssim \int_{t_1}^{t_2} |K(s, X(s))|   2^{-k-(1-\beta)M_t-2n} \]
\[
\times 2^{-k-j-2l}   \min\{  \frac{2^{k+l/2+(1+3\alpha)M_t}}{|X(s)|_{-}},
  2^{3k/2+3j+2n +l+\epsilon  M_t }  |X(s)| + 2^{k+3j+2n+2l+(1+3\epsilon) M_t}   \} d s  
 \]	
\be\label{march16eqn18}
\lesssim \int_{t_1}^{t_2}  |K(s, X(s))|  \min\{  \frac{2^{- 3M_t/2+24\alpha  M_t}}{|X(s)|_{-} },2^{ M_t/2+3\alpha  M_t}|X(s)|+ 2^{-2M_t/3+8\alpha M_t} \}  ds. 
\ee
To sum up, recall the decompositions (\ref{feb28eqn68}) and (\ref{march15eqn83}), our desired estimate (\ref{march14eqn31}) holds from the estimates (\ref{march16eqn1}), (\ref{march16eqn8}), (\ref{march16eqn9}), and (\ref{march16eqn10}). The desired estimate (\ref{march16eqn11}) holds from the estimates (\ref{march16eqn2}), (\ref{march16eqn14}), (\ref{march15eqn61}), and (\ref{march16eqn18}). 
 
\end{proof}

\begin{lemma}\label{firsttermestimate}
Under the assumptions \textup{(\ref{aproiriestimate})} and \textup{(\ref{march12eqn31})}, for  any $t_1, t_2 \in  [T^{\ast},  T^{\ast\ast} ]\subset[0,t],$  any differentiable function $C:\R_x^3 \times R_v^3\longrightarrow \mathbb{R}^3$,  
 any fixed $k\in \mathbb{Z}, j\in\mathbb{Z}_{+}$, $l\in [-j, 2]\cap \mathbb{Z}, n\in [l,2]$, s.t.,   $k\leq 40M_t, j\in [(1-\alpha)M_t,(1+\epsilon)M_t]$,  and $k\geq   -l +j + (1-3\alpha/2)M_t $,     the following estimate holds for any $m\in \{0,u,b\}$,
\be\label{march4eqn19}
\big| G_{ m}^{n;1}(C) (t_1,t_2)  \big| \lesssim \int_{t_1}^{t_2} |C(X(\tau), V(\tau))| \min \{  (|X(\tau)|_{-})^{-1} 2^{-3k/2+2j-2 l  +   2\alpha M_t }   ,  2^{-k/2+2j -l} \}d \tau.
\ee 
\end{lemma}

\begin{proof}
Recall (\ref{feb28eqn70}) and (\ref{feb28eqn72}). In terms of kernel, we have
 \[
 |G_{ m}^{n;1}(C) (t_1,t_2)| = \int_{t_1}^{t_2}  \int_{\R^3}\int_{\R^3} \int_{\R^3} K^1_{k; n }(\tau, y   ) K_{k; l,m }(z, v  ) \big(E(\tau,X(\tau) - y-z   )+ \hat{v}\times B(\tau,X(\tau) - y -z  ) \big)\]
\be\label{march3eqn93}
 \times f(\tau, X(\tau)- y -z  , v) \varphi_{j}(v)  \psi_{\leq n+4}(\angle(v,\tilde{V}(\tau ) ))  d v  d y d z d \tau, 
 \ee
where the kernel $K_{k; l,m }(z, v  )$   and  the kernel $K^1_{k; n }(y, v  )$ are defined as follows, 
\be\label{march3eqn96}
K_{k; l,m }(z, v  )= \int_{\R^3} e^{i z \cdot \xi}    \tilde{m}_{k,j,l;m} ( \xi, v )\varphi_{l;-j}(\angle(\xi, -v)) \psi_{[k-4,k+4]}(\xi) d \xi, 
\ee
\[
K^1_{k; n }(\tau, y  ):= \int_{\R^3} e^{i y \cdot \xi} 	C( X(\tau ), V(\tau ))  \big({ (|\xi|+\hat{V}( \tau   )\cdot \xi  ) } \big)^{-1}   \psi_k(\xi)   \varphi_{n;l}(\angle(\xi,-\tilde{V}(\tau )))d \xi. 
\]
In the equality (\ref{march3eqn93}),  we used the fact that the angle between $  v$ and $\tilde{V}(\tau )$ is less than $2^{n+2 } $ because the angle between $- v$ and $\xi$ is less than $2^l$ and the angle between  between $\xi$ and $-\tilde{V}( \tau )$, which is a fixed vector, is less than $2^n$. 

Recall the detailed formulas of symbols  $  \tilde{m}_{k,j,l;m} ( \xi, v ) $ in  (\ref{feb28eqn74}).  After doing integration by parts in $\xi$ in directions $\tilde{v}$ and $\tilde{v}^{\bot}$, the following   estimate holds for any $m\in \{0,e,b\},$
 \be\label{march3eqn81}
|K_{k; l,m }(z, v  )| \lesssim 2^{2k-j} (1+2^k|z\cdot \tilde{v}|)^{-1000/\epsilon}(1+2^{k+l}|z\times  \tilde{v}|)^{-1000/\epsilon}.
 \ee
After doing integration by parts in $\xi$ in directions $\tilde{V}(\tau )$ and $\tilde{V}(\tau )^{\bot}$, the following   estimate holds,
\be\label{march3eqn97}
 \big| K^1_{k; n }(\tau, y  )\big|\lesssim 2^{2k }  |C(X(\tau), V(\tau))| (1+2^k|y\cdot \tilde{V}(\tau )|)^{-100/\epsilon}(1+2^{k+n}|y\times  \tilde{V}(\tau )|)^{-100/\epsilon}.
 \ee

From the Cauchy-Schwarz inequality for the integration with respect to $z$, the conservation laws in (\ref{conservationlaw}), the boundedness of $L^\infty_{x,v}-$norm of $f$ and the volume of support of $v$. As a result, the following estimate holds from the estimates of kernels in  (\ref{march3eqn81}) and  (\ref{march3eqn97}), 
\[
 |G_{ m}^{n;1}(C) (t_1,t_2)|\]
 \[
  \lesssim \sum_{K\in \{E, B\}} \int_{t_1}^{t_2} \int_{\R^3}  \int_{\R^3}  |K^1_{k; n }(\tau, y   )|  \psi_{\leq n+4}(\angle(v,\tilde{V}(\tau ) )) \| K(\tau, X(\tau)-y-z)\|_{L^2_z} \| K_{k; l,m }(z, v  ) \|_{L^2_z} \varphi_{j}(v)     d y  d vd \tau 
\]
\be\label{march15eqn87}
\lesssim \int_{t_1}^{t_2} |C(X(\tau), V(\tau))| 2^{-k/2-j-2n-l}  2^{3j+2n }d \tau \lesssim \int_{t_1}^{t_2} |C(X(\tau), V(\tau))| 2^{-k/2+2j -l}  d \tau  .
\ee

For fixed $\tau$, based on the possible size of $|X(\tau)|$, we separate into two cases as follow. 

\noindent \textbf{Case $1$}:\quad If $|X(\tau)|\leq 2^{-k-l+2\epsilon M_t}$. \qquad From the estimate (\ref{march15eqn87}), we have
\be\label{march4eqn31}
 |G_{ m}^{n;1}(C) (t_1,t_2)| \lesssim \int_{t_1}^{t_2} |C(X(\tau), V(\tau))|  (|X(\tau)|)^{-1}   { 2^{-3k/2 +2j-2l + 2\epsilon M_t } } d \tau.
\ee

\noindent \textbf{Case $2$}:\quad If $|X(\tau)|\geq 2^{-k-l+2\epsilon M_t }$. \qquad 
For this case we have $|X(\tau) -y-z |\sim |X(\tau)| $ if $|y|+|z|\leq 2^{-k-l + \epsilon M_t }$, which is the main subcase we only have to consider. If $|y|+|z|\geq 2^{-k-l+ \epsilon M_t }$, then from the estimates of kernels in (\ref{march3eqn97}) and (\ref{march3eqn81}), we know that the size of  kernels is very small.  From the rough  estimate  of electromagnetic field (\ref{march3eqn51}) in Lemma \ref{roughestimateofelectromag}, the estimates of kernels and the volume of support of $v$, we have
\be\label{march4eqn32}
 |G_{ m}^{n;1}(C) (t_1,t_2)|  \lesssim \int_{t_1}^{t_2} |C(X(\tau), V(\tau))| \big[  2^{-2k-j-2n-2l}(|X(\tau)|_{-} )^{-1}   2^{(1+\epsilon)M_t} 2^{3j+2n} +2^{-500M_t}\big] d \tau 
\ee
 To sum up, our desired estimate (\ref{march4eqn19}) holds from the estimates (\ref{march15eqn87}),  (\ref{march4eqn31})  and (\ref{march4eqn32}). 
\end{proof}

\begin{lemma}\label{largefrequencycase}
Under the assumptions \textup{(\ref{aproiriestimate})} and \textup{(\ref{march12eqn31})},  for any  $t_1, t_2  \in [T^{\ast}, T^{\ast\ast} ]\subset[0,t],m \in\{0,e,b\}$,  the following estimate holds for any differentiable function $C:\R_x^3\times \R_v^3\longrightarrow \R^3$ s.t., $\|C(x,v)\|_{L^\infty_{x,v}}\lesssim 1$, 
\be\label{march1eqn31}
 \big| \int_{t_1}^{t_2} \int_0^s C( X(s ), V(s ))    High_{k;j,l}^{m}  (s, \tau, X(s )) d\tau d s  \big| \lesssim 	2^{-2k/7 +5M_t +4j}  \big( \min\{2^{-j}, 2^{- N_1 j + N_1 M_t }   \} \big)^{1/14}.
\ee	
\end{lemma} 
\begin{proof}
For this case, we do integration by parts in characteristic time. Recall the equality  in (\ref{feb28eqn68})  and the corresponding symbols in (\ref{feb28eqn72}), and (\ref{feb28eqn74}). Note that from the $L^\infty_x\longrightarrow L^{7/4}$ type Sobolev embedding,  the estimate of velocity of characteristics in (\ref{aproiriestimate}), and the basic estimate in (\ref{feb25eqn12}), the following estimate holds for any $  m\in \{0,e,b\},$
\[
 \sum_{i=1,2}\sum_{n\in[l,2]\cap \Z}|G_{m}^{n;i}(C)(t_1, t_2)| 
  \lesssim 2^{-j-2l} 2^{12k/7}  2^{-2k} 2^{ 2(1-\beta)M_t} \int_{0}^{t} \| \int_{\R^3} |K(s,x)| f (s,x,v) \varphi_j(v) d v \|_{L^{7/4}_x } d s 
\]
\be\label{march1eqn41}
\lesssim  2^{-j-2l} 2^{-2k/7 + 2M_t}   \int_{0}^{t} \| \int_{\R^3} 
 f (s,x,v) \varphi_j(v) d v\|_{L^{14}_x} ds\lesssim 	2^{-2k/7+2M_t+4j}  \big( \min\{2^{-j}, 2^{-  N_1 j + N_1 M_t }  \} \big)^{1/14}
\ee
Recall the detailed formula of  $ G_{m}^{n;3}(C)(t_1, t_2)$ in  (\ref{feb28eqn71}). Similarly,   by using the same strategy   used in the above estimate, the following estimate holds for  $ G_{m}^{n;3}(C)(t_1, t_2)$ from the first estimate (\ref{feb25eqn72}) in Lemma \ref{roughalongchar2},
\[
 \sum_{i=1,2}\sum_{n\in[l,2]\cap \Z}|G_{ m}^{n;3} (C)(t_1,t_2)|  \lesssim \sum_{K_1, K_2\in \{E,B\}} 2^{-j-2l} 2^{12k/7}  2^{-2k+2(1-\beta) M_t} 
   \int_{t_1}^{t_2}  |K_1(s, X(s ))| \]
\be\label{march1eqn42}
  \times  \int_{0}^{s}  \| \int_{\R^3} |K_2(\tau, x)| f (\tau,x,v) \varphi_j(v) d v \|_{L^{7/4}_x } d \tau d  s  \lesssim 	2^{-2k/7+5M_t}  2^{4j}  \big( \min\{2^{-j}, 2^{-n j+n M_t }  \} \big)^{1/14}.
\ee
Hence, our desired estimate (\ref{march1eqn31}) holds straightforwardly from the estimate (\ref{march1eqn41}) and the estimate (\ref{march1eqn42}). 
\end{proof}

\vo 

\noindent \textit{Proof of Proposition} \ref{mainproposition}.\qquad  
 
Recall the decomposition of the electromagnetic field in (\ref{march14eqn1}) and the decompositions of $K_j$ in (\ref{march16eqn41}),  (\ref{jan14eqn51}) and (\ref{feb28eqn1}).

 The desired estimate (\ref{march14eqn41}) follows from the second estimate in (\ref{feb25eqn72}) in Lemma \ref{roughalongchar2}, which is used for the case   $j\geq (1+\epsilon) M_t$,  the estimate (\ref{march4eqn1})  in Lemma \ref{roughestimateofelectromag} and the estimate (\ref{march15eqn35}) in Lemma \ref{roughestimateelectro}, which are used for the case   $j\leq (1-\alpha)  M_t$, the estimate (\ref{march2eqn40}) in Lemma \ref{nonormalform}, which is used for the case   $j\in[(1-\alpha)  M_t, (1+\epsilon)  M_t], k\leq -l+j+(1-3\alpha/2) M_t$, the  first estimate in  (\ref{march2eqn53}) in Lemma \ref{endpointcase}, the estimate (\ref{march4eqn92}) in Lemma \ref{largelandsmalln}, the estimate (\ref{march16eqn11}) in Lemma \ref{smallllargen},  which are used for the case  $j\in[(1-\alpha)  M_t, (1+\epsilon)M_t],  -l+  j+(1-3\alpha/2) M_t\leq k\leq 40 M_t$,  and the second estimate   in  (\ref{march2eqn53}) in Lemma \ref{endpointcase} and the estimate (\ref{march1eqn31}) in Lemma \ref{largefrequencycase}, which are used   for the case when $k\geq 40 M_t$.

 Similarly,  from the estimate (\ref{march4eqn1}) in Lemma \ref{roughestimateofelectromag}, the estimate (\ref{march2eqn40}) in Lemma \ref{nonormalform}, the estimate (\ref{march2eqn53}) in Lemma \ref{endpointcase}, the estimate (\ref{march16eqn3}) in Lemma \ref{largelandsmalln}, the estimate (\ref{march14eqn31}) in Lemma \ref{smallllargen}, the estimate (\ref{march1eqn31}) in Lemma \ref{largefrequencycase}, we know that our desired estimate (\ref{march16eqn51}) holds. 
 
\qed

\subsection{A pointwise bilinear estimate }

As summarized in the   Lemma \ref{bilinearestimaterough}, our main goal of this subsection is to prove  two   point-wise estimates for the bilinear form $T_{k;j}^{ l,n}(K, f)(t,s,x)$, which have been used as black boxes  previously in the proof of Lemma \ref{largelandsmalln} and Lemma \ref{smallllargen}.  These two estimates schematically give their dependence to the localized angles and the symbols of the bilinear form.

Recall (\ref{march10eqn1}). From the Kirchhoff's formulas in (\ref{dec20eqn1}) and (\ref{march4eqn100}).  We have the following formulation in terms of kernel, 
\be\label{march11eqn71}
T_{k;j}^{ l,n}(K, f)(s,\tau,x)=  M_{k;j}^{ l,n}(K, f)(s,\tau,x) +  Err_{k;j}^{ l,n}(K, f)(s,\tau,x), 
\ee
where
\[
M_{k;j}^{ l,n}(K, f)(s,\tau,x):=\int_{\R^3} \int_{\R^3} \int_{\R^3} \int_{\mathbb{S}^2}(s-\tau)  K(\tau, x-y-z+ (s-\tau)\theta)    f(\tau,x-y-z + (s-\tau)\theta, v )\]
\be\label{march10eqn11}
\times \mathcal{M}_{k;j}^{  l }(y,\theta,  v) K_{k;n}(z)\varphi_j(v) \psi_{\leq n+4}(\angle(v, -\omega_0))   d \theta d y d zd v , 
\ee
\[
Err_{k;j}^{ l,n}(K, f)(s,\tau,x)= \int_{\R^3} \int_{\R^3} \int_{\R^3} \int_{\mathbb{S}^2}   K(\tau, x-y-z+ (s-\tau)\theta)    f(\tau,x-y-z + (s-\tau)\theta, v )\]
\be\label{march10eqn41}
\times   \mathcal{E}_{k;j }^{  l }(  y, v) K_{k;n}(z)\varphi_j(v) \psi_{\leq n+4}(\angle(v,- \omega_0))  d \theta  d y d zd v, 
\ee
where we used the fact that $\angle( v, -\omega_0)\leq 2^{n+2}$ as $\angle(\xi, \omega_0)\leq 2^n$ and $\angle(\xi,- v)\leq 2^l$ due to the cutoff functions and the assumption that $n\in[l,2]\cap \Z$, and the kernels are defined as follows, 
\be\label{march11eqn1}
\mathcal{M}_{k;j}^{ l }( y, \theta,v) :=    \int_{\R^3} e^{iy\cdot\xi} i(|\xi| +\theta\cdot \xi)  m_1 ( \xi, v ) \psi_k(\xi) \varphi_{l;-j}(\angle(\xi, -v))  d \xi,  
\ee
\be\label{march11eqn2}
\mathcal{E}_{k;j}^{ l }( y, v) :=    \int_{\R^3} e^{iy\cdot\xi}  m_1 ( \xi, v ) \psi_k(\xi) \varphi_{l;-j}(\angle(\xi, -v))  d \xi,  
\quad 
K_{k;n}(z):=\int_{\R^3} e^{iy\cdot\xi}  \varphi_{n;l}(\angle(\xi, \omega_0)) m_2 ( \xi, \omega_0 )  \psi_{[k-4,k+4]}(\xi)  d \xi. 
\ee

  By  doing integration by parts in $\tilde{v}$ ($\omega_0$) direction and  $\tilde{v}^{\bot}$ ($\omega_0^{\bot}$) directions, the following estimate holds from assumptions of symbols in (\ref{symbolassumption1}) and  (\ref{symbolassumption2}),
\be\label{march10eqn10}
\sup_{\theta\in \S^2} \mathcal{M}_{k;j}^{ l }( y, \theta,v)  \lesssim   2^{4k+2l} \| m_1\|_{\mathcal{S}^\infty_{k,j,l}} (1+2^k|y\cdot \tilde{v}|)^{-10^3/\epsilon}(1+2^{k+l}|y\times  \tilde{v}|)^{- 10^3/\epsilon} .
\ee
\be\label{march10eqn50}
|\mathcal{E}_{k;j }^{l }( y, v)| \lesssim 2^{3k+2l} \| m_1\|_{\mathcal{S}^\infty_{k,j,l}} (1+2^k|y\cdot \tilde{v}|)^{-10^3/\epsilon}(1+2^{k+l}|y\times  \tilde{v}|)^{-10^3/\epsilon}. 
\ee 
\be\label{march10eqn12}
|K_{k;n}(z)| \lesssim    2^{3k+2n} \| m_2\|_{\mathcal{S}^\infty_{k;n,l }} (1+2^k|z\cdot {\omega}_0|)^{-10^3/\epsilon}(1+2^{k+n}|z\times  {\omega}_0|)^{-10^3/\epsilon} .
\ee

\begin{lemma}\label{bilinearestimaterough}
For any fixed   $x\in \R^3/\{0\},$ $s\in  [0,t],$ $k\in \mathbb{Z}, j\in\mathbb{Z}_{+}$, $l \in [-j, 2]\cap \mathbb{Z}, n\in [l, 2]\cap \mathbb{Z}$, s.t.,    $k\leq 40M_t, j\in [(1-\alpha)M_t,(1+\epsilon) M_t]$,   $k\geq -l +j + (1-3\alpha/2)M_t$,    the following estimate holds for the bilinear form defined in \textup{(\ref{march10eqn1})},
\be\label{march10eqn2}
\sum_{K\in \{E, B\}}\int_0^s   \big|T_{k;j}^{ l,n}(K, f)(s,\tau,x)\big| d \tau\lesssim  \frac{ \| m_1\|_{\mathcal{S}^\infty_{k,j,l}} \| m_2\|_{\mathcal{S}^\infty_{k;n,l}}}{2^n +| \tilde{x}\times \omega_0|  }  \frac{ 2^{k+l/2+ n+(1+3\alpha)  M_t } }{r_{-}}, \quad r:=|x|. 
\ee
Moreover, the following rough estimate also holds,
\be\label{march14eqn100}
 \sum_{K\in \{E, B\}}\int_0^s   \big|T_{k;j}^{ l,n}(K, f)(s,\tau,x)\big| d \tau\lesssim \| m_1\|_{\mathcal{S}^\infty_{k,j,l}}   \| m_2\|_{\mathcal{S}^\infty_{k;n,l }}\big[ 2^{3k/2+3j+2n +l+\epsilon  M_t }  r  + 2^{k+3j+2n+2l+(1+3\epsilon) M_t}\big] .
\ee
\end{lemma}
\begin{proof}
 
Let $ \tilde{\delta}:=2^{-k-l+\epsilon  M_t}$. Based on the possible size of $t-s$, we separate into two cases as follows.

\textbf{Case $1$}:\quad If $|s-\tau|\leq  \tilde{\delta}$, i.e., $\tau\in [s-\tilde{\delta}, s]$  .\qquad  Recall (\ref{march10eqn1}).  Note that after using the Cauchy-Schwarz inequality for the integration with respect to $\xi$ and the volume of support of $\xi$ and $v$, the following estimate holds if $r:=|x|\leq \tilde{\delta} 2^{ \epsilon  M_t}$, 
\[
|  T_{k;j}^{ l,n}(K, f)(s,\tau,x) |\lesssim   \| m_1\|_{\mathcal{S}^\infty_{k,j,l}}  \| m_2\|_{\mathcal{S}^\infty_{k;n,l }} 2^{(3k+2l)/2} \int_{\R^3}\|K(\tau,x)f(\tau,x,v)\|_{L^2_x } \varphi_j(v)\psi_{\leq n+4}(\angle(-v, \omega_0))  dv
\]
\be\label{march10eqn13}
 \lesssim   \| m_1\|_{\mathcal{S}^\infty_{k,j,l}}  \| m_2\|_{\mathcal{S}^\infty_{k;n,l }}2^{3k/2+3j+2n+l } \lesssim   \| m_1\|_{\mathcal{S}^\infty_{k,j,l}} \| m_2\|_{\mathcal{S}^\infty_{k;n,l }} \frac{2^{ k/2+3 j+2n +2\epsilon M_t }}{r}.
\ee

 Note that, $|x-y-z+(s-\tau)\theta|\sim |x|$ if $|y|+|z|\leq \tilde{\delta}$ and $|x|\geq \tilde{\delta} 2^{ \epsilon  M_t}$. Moreover, from the estimates of kernels in  (\ref{march10eqn10}--\ref{march10eqn12}), we know that   the kernel is extremely small if $|y|+|z|\geq  \tilde{\delta}$. Therefore, from the rough estimate of electromagnetic field (\ref{march3eqn51}) in Lemma \ref{roughestimateofelectromag} and the volume of support of $v$,  the following estimate holds if  $r:=|x|\geq \tilde{\delta} 2^{ \epsilon  M_t}$,
\be\label{march10eqn14}
|T_{k;j}^{ l,n}(K, f)(s,\tau,x)|\lesssim   \| m_1\|_{\mathcal{S}^\infty_{k,j,l}}  \| m_2\|_{\mathcal{S}^\infty_{k;n,l }} \big[   (s-\tau) 2^k  \frac{2^{(1+\epsilon)M_t}}{r} 2^{3j+2n} + 2^{-10M_t }  \|\int_{\R^3}|K(\tau,x)|f(\tau,x,v)\varphi_j(v) d v \|_{L^1_x}\big].
\ee
Combining the estimate (\ref{march10eqn13}) and the estimate  (\ref{march10eqn14}), we have
\[
\int_{s-\tilde{\delta}}^s  |  T_{k;j}^{ l,n}(K, f)(s,\tau,x)|ds \lesssim   \| m_1\|_{\mathcal{S}^\infty_{k,j,l}} \| m_2\|_{\mathcal{S}^\infty_{k;n,l }} \big[ \frac{2^{- k/2+3j+2n -l+3\epsilon M_t }}{r} +\frac{ 2^{- k +3j+2n -2l+(1+6\epsilon) M_t }}{r} + 2^{-5M_t}\big]
\]
\be\label{march10eqn15}
\lesssim  \| m_1\|_{\mathcal{S}^\infty_{k,j,l}}   \| m_2\|_{\mathcal{S}^\infty_{k;n,l }}  \frac{   2^{k+l/2+ 2n+(1+10\alpha)  M_t/2 } }{r_{-}} .
\ee

\textbf{Case $2$}:\quad If $|s-\tau|\geq \tilde{\delta}$, i.e., $\tau\in [0,s-\tilde{\delta}]$. \qquad

To better see the angular relations, we localize the angle of $v$ by using the following partition of unity. 
Let $\eta_i(\omega), \omega \in \mathbb{S}^2, i\in \{1,\cdots, K\}$ be  a labeled 	partition of unit  for the unit  sphere such that the support of $\eta_i(\cdot)$ is contained in small ball on sphere with radius of size $2^l$ and the supports of $\eta_i(\cdot)$ overlaps only finite times.  More precisely, we have
\be\label{march11eqn140}
\forall \omega\in \mathbb{S}^2, \quad 1=\sum_{i=1,\cdots, L} \eta_i(\omega), \quad  |supp(\eta_i(\cdot))|\lesssim 2^{2l},\quad  |L|\sim 2^{-2l}, \quad \eta_i(\omega)\geq 0,\quad  \eta_i(\omega)\eta_j(\omega)=0\,\,\textup{if}\,\, |i-j|\geq C. 
\ee
Moreover, we fix a choice of   $\{\omega_i\}_{i=1}^L \subset \mathbb{S}^2$. s.t., $\omega_i \in supp(\eta_i(\cdot))$. Therefore, $ supp(\eta_i(\cdot))\subset supp(\psi_{\leq l+10}(\cdot -\omega_i))$.

Once we localize $\tilde{v}$ inside  $ supp(\eta_i(\cdot))$ , due to the cutoff function $\varphi_{l;-j}(\angle(-v, \xi))$, we know that $\xi$ is also  localized  in a sector of size $2^l$ centered at $\omega_i$. Moreover,    due to the cutoff function $\psi_{\leq n+4}(\angle(\xi,\omega_0)$, where $\omega_0\in \mathbb{S}^2$ is fixed,  we know that    there are at most $2^{2n-2l}$ sectors on sphere   to be  considered.

Recall the integral (\ref{march4eqn41}) in the proof of Kirchhoff's formula. From the stationary phase point of view, we know that $\theta$ is   localized near $\xi/|\xi|$ and $-\xi/|\xi|$ with radius of size $(|s-\tau||\xi|)^{-1/2}$. Hence, it is also localized roughly  near the fixed direction $\omega_0$ with radius of size $(|s-\tau||\xi|)^{-1/2} +2^n$. 

 Motivated from the above discussion, we  define $\bar{\theta}_{s}^l:=(1+|s| 2^k)^{-1/2 }2^{\epsilon M_t} +2^{l+20}$ and the following two cutoff functions, 
\be\label{march11eqn11}
\varphi_{ess;i}(s,  \theta):= \sum_{\mu\in\{+,-\}} \psi(\angle(\theta, \mu \omega_i)/\bar{\theta}_{s}^l) \psi(\angle(\theta, \mu \omega_0)/\bar{\theta}_{s}^n), \quad \varphi_{gd }(s,  \theta)= 1-  \sum_{i=1,\cdots L }  \varphi_{ess;i}(s,  \theta).
\ee 

Recall  (\ref{march10eqn11}). After using the partition of unity in (\ref{march11eqn140}) for the direction of $v$, $\tilde{v}$, and the partition of unity in (\ref{march11eqn11}) for $\theta$, we have
\be\label{march11eqn58}
M_{k;j}^{ l,n}(K, f)(s,\tau,x)=   M_{k;j}^{ l,n;gd}(K, f)(s,\tau,x)   + \sum_{i=1,\cdots K_n}  M_{k;j}^{ l,n;ess,i}(K, f)(s,\tau,x) ,\quad K_n\sim 2^{2n-2l}
\ee
where  the angular localized $ M_{k;j }^{l,n;u}(K, f)( s,\tau,x), u\in\{gd,ess;i\},$ in physical space are defined as follows, 
\[
M_{k;j}^{ l,n;u}(K, f)(s,\tau,x) := \int_{\R^3} \int_{\R^3}    \int_{\R^3} \int_{\mathbb{S}^2}(s-\tau)  \mathcal{M}_{k;j }^{ l  }(y,\theta,  v) K_{k;n}(z)  \psi_{\leq n+4}(\angle(-v, \omega_0))       \big) \varphi_j(v)  \varphi_{u }(s-\tau ,  \theta)   \eta_i(\tilde{v})\]
 \be\label{march10eqn30}
 \times  K(\tau, x-y-z+ (s-\tau)\theta) f(\tau,x-y -z+ ( s-\tau)\theta, v )    d y d z d v d \theta,\quad u\in\{gd,ess;i\}.
\ee

Let $\xi$ be fixed. After changing coordinates such that $\xi\cdot\theta=|\xi|\cos\sigma$, we  do integration by parts in $\sigma $ many times. As a result, the following estimate holds for any $N\in \mathbb{Z}_{+}$ if $|s-\tau|\geq 2^{-k+\epsilon  M_t}$, 
\be 
  \big|\int_{\mathbb{S}^2} e^{i  (s-\tau) \xi\cdot \theta} \varphi_{gd }(s-\tau,  \theta) d \theta\big|\lesssim_N  2^{-N M_t }. 
\ee
Therefore, from the above estimate and the volume of support of $\xi$, the following estimate holds for any $\tau\in [0,s-\tilde{\delta}]$,
\[
 \big|  M_{k;j }^{ l,n;gd}(K, f)(s,\tau,x) \big|   \lesssim  \| m_1\|_{\mathcal{S}^\infty_{k,j,l}}   \| m_2\|_{\mathcal{S}^\infty_{k;n,l }}  2^{-200M_t} 2^{(3k+2l)/2}  \int_{\R^3}\| |K(\tau,z)| f(\tau,z,v)\|_{L^2_z } \varphi_j(v) dv
\]
\be\label{march11eqn59}
 \lesssim 2^{-10M_t}  \| m_1\|_{\mathcal{S}^\infty_{k,j,l}}   \| m_2\|_{\mathcal{S}^\infty_{k;n,l }} . 
\ee

Now we focus on the essential part. We first rule out the case   $|y|+|z|\geq \tilde{\delta}2^{-\epsilon M_t/2}=2^{-k-l+\epsilon M_t/2}$, in which the kernels provide sufficiently fast decay. Note that, from the rough estimate of the electromagnetic field (\ref{march3eqn51}) in Lemma \ref{roughestimateofelectromag} and the estimates of kernels in (\ref{march10eqn10}) and (\ref{march10eqn12}), we have 
\[
  \big| \int_{\R^3} \int_{\mathbb{S}^2}\int_{|y|+ |z|\geq \tilde{\delta}2^{-\epsilon M_t/2} }   (s-\tau)    \mathcal{M}_{k;j }^{ l  }(y,\theta,  v) K_{k;n}(z)  \psi_{\leq n+4}(\angle(-v, \omega_0))       \big) \varphi_j(v)  \varphi_{u }(s-\tau ,  \theta)   \eta_i(\tilde{v})\]
\[
 \times  K(\tau, x-y-z+ (s-\tau)\theta) f(\tau,x-y -z+ ( s-\tau)\theta, v )    d y d z  d \theta d v\big|
\]
\[
\lesssim  \| m_1\|_{\mathcal{S}^\infty_{k,j,l}} \| m_2\|_{\mathcal{S}^\infty_{k;n,l}} \int_{\R^3} \int_{\mathbb{S}^2}\int_{|y|+ |z|\geq \tilde{\delta} 2^{-\epsilon M_t/2} } (s-\tau) \frac{ 2^{4k+2l}2^{3k+2n} (1+2^{k+l}|y|)^{-10^3/\epsilon}(1+2^{k+n}|z|)^{-10^3/\epsilon}}{|x-y-z +(s-\tau)\theta|_{-}} 
\]
\be\label{march24eqn110}
\times 2^{3j+(1+\epsilon)M_t} d y d z d \theta  \lesssim \| m_1\|_{\mathcal{S}^\infty_{k,j,l}} \| m_2\|_{\mathcal{S}^\infty_{k;n,l}}.
\ee

Recall the estimate of kernel $K_{k;n}(z)$ in  (\ref{march10eqn12}).  Note that, the following estimate holds from the above estimate (\ref{march24eqn110}) and  the Cauchy-Schwarz inequality for fixed  $z$,  
\[
 \sum_{i=1, \cdots, K_n}  \big|  M_{k;j }^{l,n;ess,i}(K, f)(s,\tau,x)  \big| \lesssim   \| m_2\|_{\mathcal{S}^\infty_{k;n,l }}  \big[\| m_1\|_{\mathcal{S}^\infty_{k,j,l}} 
\]
 \be\label{march10eqn21}
  +\int_{|z|\leq \tilde{\delta} 2^{-\epsilon M_t/2} } (s-\tau)  2^{3k+2n} (1+2^k|z\cdot {\omega}_0|)^{-10^3/\epsilon}(1+2^{k+n}|z\times  {\omega}_0|)^{-10^3/\epsilon} F(s,\tau,x,z) d z  \big],
 \ee
 where 
 \[
0\leq F(s,\tau,x,z)\]
\[
 \lesssim \big[ \sum_{i=1, \cdots, K_n}   \int_{|y|\leq\tilde{\delta} 2^{-\epsilon M_t/2} }\int_{\mathbb{S}^2}  \int_{\R^3} |    \mathcal{M}_{k;j }^{ l  }(y,\theta,  v)| |K(\tau, x-y-z+(s-\tau)\theta)|^2 \eta_i(\tilde{v}) \varphi_{ess;i}(s-\tau,  \theta) \varphi_j(v) d v d \theta  dy \big]^{1/2}
 \]
 \be\label{march10eqn22}
 \times \big[ \sum_{i=1,\cdots K_n}   \int_{|y|\leq\tilde{\delta}2^{-\epsilon M_t/2}}\int_{\mathbb{S}^2}  \int_{\R^3}  |  \mathcal{M}_{k;j }^{ l  }( y, \theta,v)| |f(\tau, x-z-y+(s-\tau)\theta, v)|  \eta_i(\tilde{v}) \varphi_{ess;i}(s-\tau,  \theta) \varphi_j(v) d v d \theta  dy  \big]^{1/2}, 
 \ee

Note that,   
\[
\textup{if\,\,} |\tilde{v}-\theta|\geq 2^5 \bar{\theta}_{s-\tau}^l\,\quad \Longrightarrow\quad (\tilde{v}, \theta)\notin \cup_{i=1,\cdots, L} supp(\eta_i(\tilde{v})\varphi_{ess;i}(t,  \theta))\subset \mathbb{S}^2\times \mathbb{S}^2. 
\]
With the above fact,   from the estimate of kernel in  (\ref{march10eqn10}) and the estimate  (\ref{march6eqn1}) in Lemma \ref{basicestimateint} and the point-wise estimate of electromagnetic field (\ref{march3eqn51}) in Lemma \ref{roughestimateofelectromag} and the volume of support of $v$,  the following estimate holds for any fixed $z\in \R^3$, 
\[
 \sum_{i=1, \cdots, K_n} \sum_{K\in \{E, B\}}   \int_{|y|\leq\tilde{\delta}2^{-\epsilon M_t/2}} \int_{\mathbb{S}^2}  \int_{\R^3}   |\mathcal{M}_{k;j }^{ l  }( y, \theta,v)| |K(\tau, x-y-z+(s-\tau)\theta)|^2 \eta_i(\tilde{v}) \varphi_{ess;i}(s-\tau,  \theta) \varphi_j(v) d v d \theta  dy 
\]
 \be\label{march10eqn36}
\lesssim      \int_{|y|\leq\tilde{\delta}2^{-\epsilon M_t/2}} 2^{4k+2l} \| m_1\|_{\mathcal{S}^\infty_{k,j,l}}    \frac{   \bar{\theta}_{s-\tau}^n 2^{3j  } \big( \bar{\theta}_{s-\tau}^l \big)^2 \min\{| |x-y-z|-(s-\tau)|^{-1},2^{(1+\epsilon)M_t} |s-\tau| \}}{(\bar{\theta}_{s-\tau}^n  +| \tilde{x}\times \omega_0|)|x-y-z|(s-\tau) }  d y. 
\ee

 Moreover, after changing coordinates for $y$ in terms of the  cylinder coordinates system with $\tilde{v}$ as axis as in (\ref{marcheqn16eqn69}),  from the estimate of kernel in  (\ref{march10eqn10}), the following   estimate holds after using the estimate (\ref{feb29eqn1}) in Lemma \ref{basicestimateint} or using  the volume of support of $\theta$ and $v$, 
\[
 \sum_{i=1,\cdots K_n}   \int_{|y|\leq\tilde{\delta} 2^{-\epsilon M_t/2} } \int_{\mathbb{S}^2}  \int_{\R^3}    | \mathcal{M}_{k;j }^{ l  }( y, \theta,v)|  |f(\tau, x-z-y+(s-\tau)\theta, v)|  \eta_i(\tilde{v}) \varphi_{ess;i}(s-\tau,  \theta) \varphi_j(v) d v d \theta  dy 
\]
\[
\lesssim  \sum_{i=1,\cdots K_n}\int_{\R} \int_{\R}   \int_{\R^3}  \int_{\mathbb{S}_v }  \int_{\mathbb{S}^2}     2^{4k+2l} \| m_1\|_{\mathcal{S}^\infty_{k,j,l}}  (1+2^k a )^{-10}(1+2^{k+l}r  )^{-10}  \eta_i(\tilde{v}) \varphi_{ess;i}(s-\tau,  \theta) \varphi_j(v)      \]
\[
\times  f(\tau, x-z- a\tilde{v}-r\omega_v +(s-\tau)\theta, v )  d\theta   d \omega_v d v rdr da
\]
 \be\label{march10eqn37} 
\lesssim \| m_1\|_{\mathcal{S}^\infty_{k,j,l}} \min\{ \frac{  2^{k } C_n(x, \omega_0)   \bar{\theta}_{s-\tau}^n 2^{-j}  }{|x-z|(s-\tau) | |x-z|-(s-\tau)|},   |K_n| 2^{k}   2^{3j+2l}\big( \bar{\theta}_{s-\tau}^{l}\big)^2\},
\ee
where $C_n(x, \omega_0) :=  ({2^n +| \tilde{x}\times \omega_0|})^{-1}$.

With the above preparations,  based on the possible size of ``$r$'' and ``$\tau$'', we separate into four sub-cases as follows.

\textbf{Subcase $1$}:\quad If $r\leq \tilde{\delta}2^{ \epsilon M_t}$  and  $\tau\in [s-r- \tilde{\delta}, s-r+ \tilde{\delta}]\cap[0,s-\tilde{\delta}]$. \qquad Recall (\ref{march10eqn30}). Note that we have $|  s-\tau|\lesssim \tilde{\delta}2^{\epsilon   M_t}$ for the case we are considering.  From the Cauchy-Schwarz inequality for the integration with respect to $y$, the estimate of kernels in (\ref{march10eqn10}) and (\ref{march10eqn12}), and the volume of support of $\theta$ and $v$, we have
\[  
\sum_{i=1,\cdots K_n}   \int_{ s-r-  \tilde{\delta}}^{ \min\{s-\tilde{\delta}, s-r+  \tilde{\delta}\}} \big|  M_{k;j}^{ l,n;ess,i}(K, f)(s,\tau,x) \big| d \tau\]
\[
 \lesssim   \| m_1\|_{\mathcal{S}^\infty_{k,j,l}}  \| m_2\|_{\mathcal{S}^\infty_{k;n,l }}    \int_{ s-r-  \tilde{\delta}}^{ s-r+  \tilde{\delta}} (s-\tau) |K_n| 2^{5k/2+l   }\big( \bar{\theta}_{s-\tau}^{l}\big)^2 2^{3j+2l}    d \tau
\]
\be\label{march11eqn15}   
 \lesssim  \| m_1\|_{\mathcal{S}^\infty_{k,j,l}}  \| m_2\|_{\mathcal{S}^\infty_{k;n,l }}   2^{5k/2 +3j +2n  + l+6\epsilon   M_t} 2^{-2k-l  } \lesssim \| m_1\|_{\mathcal{S}^\infty_{k,j,l}}  \| m_2\|_{\mathcal{S}^\infty_{k;n,l }}  {    2^{ k  +4\alpha M_t +2n  +l/2 }  } r^{-1}.
 \ee

\textbf{Subcase $2$}:\quad  If $r\geq \tilde{\delta}2^{ \epsilon  M_t}$  and  $\tau\in [s-r- \tilde{\delta}, s-r+ \tilde{\delta}]\cap[0,s-\tilde{\delta}]$. 

For this case we are considering,    we have $|s-\tau|\sim r  $ and $|x-y-z|,|x-z|\sim r $   
 if $|y|+|z|\leq \tilde{\delta}=2^{-k-l+\epsilon  M_t}$.  From the estimate  (\ref{march10eqn22}) and the estimates (\ref{march10eqn36}) and (\ref{march10eqn37}), the following estimate holds for fixed $z\in \R^3$ s.t., $|z|\leq \tilde{\delta},$
  \[
\int_{ s-r-  \tilde{\delta}}^{ s-r+  \tilde{\delta}}  (s-\tau) F(s,\tau,x,z) d \tau\lesssim  \| m_1\|_{\mathcal{S}^\infty_{k,j,l}} \int_{ s-r-  \tilde{\delta}}^{ s-r+  \tilde{\delta}}  (s-\tau)   \big( \min\{ 2^{3j+2n+k} \big( \bar{\theta}_{s-\tau}^{l}\big)^2, \frac{ C_n(x, \omega_0)  2^{k-j} \bar{\theta}_{s-\tau}^n}{r^2||x-z|-(s-\tau)|}	\} \big)^{1/2} \]
\[ 
\times  \big(\int_{|y|\leq \tilde{\delta} 2^{-\epsilon M_t/2} }  \frac{2^{2\epsilon  M_t}   C_n(x, \omega_0) \bar{\theta}_{s-\tau}^n 2^{4k+2l+3j}\big( \bar{\theta}_{s-\tau}^l \big)^2  }{r^2||x-y-z|-(s-\tau)|^{1-\epsilon}}  dy   \big)^{1/2}   d\tau
\]
\[
\lesssim\| m_1\|_{\mathcal{S}^\infty_{k,j,l}}\big[ C_n(x, \omega_0)  {2^{5 k/2 + l+3j/2+\epsilon    M_t }}\big( \int_{ s-r-  \tilde{\delta}}^{ s-r+  \tilde{\delta}}\int_{|y|\leq \tilde{\delta} 2^{-\epsilon M_t/2} } \frac{ 1}{ ||x-y-z|-(s -\tau)|^{1-\epsilon}}  dy d\tau    \big)^{1/2}\] 
\[
\times   \big( \int_{ s-r-  \tilde{\delta}}^{ s-r+  \tilde{\delta}} \min\{  2^{ 2n+3j} \big( \bar{\theta}_{s-\tau}^{ l}\big)^4 \bar{\theta}_{s-\tau}^n  , \frac{2^{ -j}(\bar{\theta}^l_{s-\tau})^2 ( \bar{\theta}_{s-\tau}^n)^2}{r^2||x-z|-(s-\tau)|} \} d \tau \big)^{1/2} \big]\lesssim  \| m_1\|_{\mathcal{S}^\infty_{k,j,l}}  C_n(x, \omega_0) 2^{k+3j/2-l/2+4\epsilon    M_t }
\]
\be\label{march11eqn31}
\times \big(\min\{  \tilde{\delta}2^{3j+2n}(2^{ 4l+n} +  r^{-2}2^{-2k+n} +r^{-5/2} 2^{-5k/2} ), r^{-2} 2^{-j+2l+2n}+ r^{-3} 2^{-j-k+2n} + r^{-4} 2^{-j-2k}\}\big)^{1/2}    
\ee
\be\label{march11eqn32}
\lesssim  \| m_1\|_{\mathcal{S}^\infty_{k,j,l}}C_n(x, \omega_0)2^{k+l/2+n+(1+3\alpha)   M_t} r^{-1}.
\ee
In the above final estimate,  we used the second estimate in (\ref{march11eqn31}) if $r\geq 2^{-k-2l}$ and used the first estimate  in (\ref{march11eqn31}) if $\tilde{\delta}2^{ \epsilon M_t}\leq r\leq 2^{-k-2l}$. 
Therefore, after combining the estimates (\ref{march10eqn21}) and  (\ref{march11eqn32}), we have
\be\label{march10eqn93}
  \sum_{i=1, \cdots, K_n}   \int_{ s-r-  \tilde{\delta}}^{ \min\{s-\tilde{\delta}, s-r+  \tilde{\delta}\}}  \big|  M_{k;j}^{ l,n;ess,i}(K, f)(s,\tau,x) \big| d\tau   \lesssim    C_n(x, \omega_0)   \| m_1\|_{\mathcal{S}^\infty_{k,j,l}}  \| m_2\|_{\mathcal{S}^\infty_{k;n,l }}  2^{k+l/2+n+(1+3\alpha) M_t} r_{-}^{-1}.
\ee
 
\textbf{Subcase $3$}:\quad  If $r\geq \tilde{\delta}2^{ -\epsilon M_t/3}$  and  $\tau\in [s-r- \tilde{\delta}, s-r+ \tilde{\delta}]^c\cap[0,s-\tilde{\delta}]$.\qquad For this case, we have $|x-y-z|, |x-z|\sim r $ and $||x-y|-(s-\tau)|\sim ||x|-(s-\tau)|$ if $|y|,|z|\leq 2^{-k- l+\epsilon M_t/2 }=\tilde{\delta}2^{-\epsilon M_t/2} $ .    From the estimate  (\ref{march10eqn22}) and the estimates (\ref{march10eqn36}) and (\ref{march10eqn37}), the following estimate holds for fixed $z\in \R^3$ s.t., $|z|\leq \tilde{\delta},$
  \[
  \int_{[0, s-\tilde{\delta} ]\cap [s-r-\tilde{\delta} , s-r+ \tilde{\delta} ]^c}    (s-\tau) F(s,\tau,x,z) d \tau \lesssim  C_n(x, \omega_0) \| m_1\|_{\mathcal{S}^\infty_{k,j,l}}  \big[\int_{[0, s-\tilde{\delta} ]\cap [s-r-\tilde{\delta} , s-r+ \tilde{\delta} ]^c} 
 (s-\tau)\]
\be\label{march10eqn61}
\times  \min\{ \frac{2^{k+j- l/2}(\bar{\theta}_{s-\tau}^l) \bar{\theta}_{s-\tau}^n  }{(r(s-\tau)|r-(s-\tau)|) }, \frac{2^{k+3j- l/2+ n}(\bar{\theta}^l_{s-\tau})^2   (\bar{\theta}_{s-\tau}^n )^{1/2} }{(r(s-\tau)|r-(s-\tau)|)^{1/2}}\} d s  \big].
 \ee
  If $|s-\tau|2^k\geq 2^{-2l} $, i.e., $\bar{\theta}_{s-\tau}^l\leq 2^{l+\epsilon M_t}$, then we use the first estimate in (\ref{march10eqn61}). Meanwhile, if  $|s-\tau|2^k\leq 2^{-2l}$, i.e., $\bar{\theta}_{s-\tau}^l\leq (s-\tau)^{-1/2}2^{-k/2+\epsilon M_t}$, then we use the second estimate in (\ref{march10eqn61}). As a result, we have
  \[ 
   \int_{[0, s-\tilde{\delta} ]\cap [s-r-\tilde{\delta} , s-r+ \tilde{\delta} ]^c}    (s-\tau) F(s,\tau,x,z) d \tau    \lesssim  C_n(x, \omega_0) \| m_1\|_{\mathcal{S}^\infty_{k,j,l}}  2^{k+l/2+n+(1+3 \alpha)M_t} r_{-}^{-1}. 
  \]
  After combining it with the estimate (\ref{march10eqn21}),  we have
\be\label{march11eqn43}
    \sum_{i=1, \cdots, K_n}    \int_{[0, s-\tilde{\delta} ]\cap [s-r-\tilde{\delta} , s-r+ \tilde{\delta} ]^c} \big|  M_{k;j}^{ l,n;ess,i}(K, f)(s,\tau,x) \big| d\tau \lesssim     C_n(x, \omega_0)   \| m_1\|_{\mathcal{S}^\infty_{k,j,l}}  \| m_2\|_{\mathcal{S}^\infty_{k;n,l }}  2^{k+l/2+n+(1+ 3\alpha) M_t} r_{-}^{-1}.
  \ee

 \textbf{Subcase $4$}:\quad   If $r\leq \tilde{\delta}2^{ -\epsilon M_t/3}$  and  $\tau\in [s-r- \tilde{\delta}, s-r+ \tilde{\delta}]^c\cap[0,s-\tilde{\delta}]$.\qquad  For this case, we have $||x-y-z|-(s-\tau)|\sim |s-\tau|$   if $|y|+|z|\leq 2^{-k- l+\epsilon  M_t/2 }=\tilde{\delta}2^{-\epsilon M_t/2} $. Recall (\ref{march10eqn30}).  From the estimates of kernels in (\ref{march10eqn10}) and (\ref{march10eqn12}),  the rapid decay rate of kernel if $|y|\geq 2^{-k- l+\epsilon M_t/2 }=\tilde{\delta}2^{-\epsilon M_t/2} $, the point-wise estimate of electromagnetic field (\ref{march3eqn51}) in Lemma \ref{roughestimateofelectromag}  and the estimate (\ref{march10eqn37}), the following estimate holds,  
  \[
    \sum_{i=1, \cdots, K_n}    \int_{[0, s-\tilde{\delta} ]\cap [ s-r-\tilde{\delta} , s-r+ \tilde{\delta} ]^c} \big|  M_{k;j}^{ l,n;ess,i}(K, f)(s,\tau,x) \big| d\tau \]
    \[
     \lesssim   \| m_1\|_{\mathcal{S}^\infty_{k,j,l}}   \| m_2\|_{\mathcal{S}^\infty_{k;n,l }}  \big[  \int_{|z|\leq  2^{-k-n+\epsilon M/2}}  \int_{[0, s-\tilde{\delta} ]\cap [s-r-\tilde{\delta} , s-r+ \tilde{\delta} ]^c}  2^{3k+2n}   (1+2^k|z\cdot {\omega}_0|)^{-100/\epsilon}(1+2^{k+n}|z\times  {\omega}_0|)^{-100/\epsilon}   \]
 \be\label{march10eqn71}
 \times (s-\tau) \frac{2^{(1+\epsilon)M_t}}{|s-\tau|}    \min\{ \frac{  2^{k } C_n(x, \omega_0)   \bar{\theta}_{s-\tau}^n 2^{-j}  }{|x-z|(s-\tau) | |x-z|-(s-\tau)|},   |K_n| 2^{k}   2^{3j+2l}\big( \bar{\theta}_{s-\tau}^{l}\big)^2\}  d \tau  d z  +2^{-10M_t}\big].
\ee 
If $s-\tau\leq 2^{-k-2l+\alpha M_t }$  then from the second estimate of (\ref{march10eqn71}), we have
\be\label{march11eqn50}
\textup{(\ref{march10eqn71})} \lesssim  \| m_1\|_{\mathcal{S}^\infty_{k,j,l}}   \| m_2\|_{\mathcal{S}^\infty_{k;n,l }}    2^{k+4(1+\epsilon)M_t+2n}\big(2^{-k-2l+\alpha M_t }2^{2l} + 2^{-k} \big) r^{-1}2^{-k-l+\epsilon M_t }  \lesssim r^{-1} 2^{k + l +2 n+6\alpha M_t},
\ee
If $s-\tau\geq 2^{-k-2l+\alpha M_t }$,( i.e., $ \bar{\theta}_{s-\tau}^{l}\sim 2^l$),    and $ 2^{-k-n+2\epsilon M_t}\leq r \leq \tilde{\delta}2^{\epsilon M_t}$,  then we have $|x-z|\sim r$ if $|z|\leq 2^{-k-n+\epsilon M/2} $. From the  geometric mean   of two estimates in  (\ref{march10eqn71}),  we have
\[
\textup{(\ref{march10eqn71})} \lesssim  C_n(x, \omega_0)  \| m_1\|_{\mathcal{S}^\infty_{k,j,l}}  \| m_2\|_{\mathcal{S}^\infty_{k;n,l }}   r^{-1/2} 2^{k+2(1+\epsilon)M_t+l+3n/2}
\]
\be\label{march11eqn52}
\lesssim  C_n(x, \omega_0) \| m_1\|_{\mathcal{S}^\infty_{k,j,l}}  \| m_2\|_{\mathcal{S}^\infty_{k;n,l }}   r^{-1} 2^{k+(1+2\alpha)M_t+l+3n/2}. 
\ee
If $s-\tau\geq 2^{-k-2l+\alpha M_t }$,( i.e., $ \bar{\theta}_{s-\tau}^{l}\sim 2^l$),   and $ r\leq 2^{-k-n+2\epsilon M_t} $,   then from the  geometric mean   of two estimates in  (\ref{march10eqn71}),  we have
\[
\textup{(\ref{march10eqn71})} \lesssim  C_n(x, \omega_0)  \| m_1\|_{\mathcal{S}^\infty_{k,j,l}}   \| m_2\|_{\mathcal{S}^\infty_{k;n,l }}   \int_{|z|\leq 2^{-k-n+2\epsilon M_t}} 2^{3k+2n}  |x-z|^{-1/2} 2^{k+2(1+\epsilon)M_t+l+3n/2}  d z\lesssim  C_n(x, \omega_0)  
\]
\be\label{march11eqn53}
\times    \| m_1\|_{\mathcal{S}^\infty_{k,j,l}}   \| m_2\|_{\mathcal{S}^\infty_{k;n,l }}    2^{k+2(1+\epsilon)M_t+l+3n/2} 2^{k/2-n/2+6\epsilon M_t} \lesssim C_n(x, \omega_0)  \| m_1\|_{\mathcal{S}^\infty_{k,j,l}}  \| m_2\|_{\mathcal{S}^\infty_{k;n,l }}  r^{-1}2^{k+ (1+2\alpha)M_t +3l/2  }.
\ee
To sum up, after combining the estimates (\ref{march11eqn50}--\ref{march11eqn53}),  in whichever case  we have 
\be\label{march11eqn54}
    \sum_{i=1, \cdots, K_n}    \int_{[0, s-\tilde{\delta} ]\cap [ s-r-\tilde{\delta} , s-r+ \tilde{\delta} ]^c} \big|  M_{k;j}^{ l,n;ess,i}(K, f)(s,\tau,x) \big| d\tau \lesssim C_n(x, \omega_0) \| m_1\|_{\mathcal{S}^\infty_{k,j,l}} \| m_2\|_{\mathcal{S}^\infty_{k;n,l }}   r^{-1} 2^{k+(1+3\alpha) M_t+l/2+ n }. 
\ee

 Recall the decomposition (\ref{march11eqn58}). To sum up, after combining the estimates (\ref{march11eqn59}), (\ref{march11eqn15}),  (\ref{march10eqn93}), (\ref{march11eqn43}), and  (\ref{march11eqn54}), we have
\be\label{march11eqn61}
 \int_0^{s-\tilde{\delta}}\big| M_{k;j}^{ l,n}(K, f)(s,\tau,x)\big| d \tau \lesssim  C_n(x, \omega_0) \| m_1\|_{\mathcal{S}^\infty_{k,j,l}}  \| m_2\|_{\mathcal{S}^\infty_{k;n,l }} r_{-}^{-1} 2^{k+(1+3\alpha) M_t+l/2+ n }.
 \ee

 Recall the detailed formulas of $M_{k;j}^{ l,n}(K, f)(s,\tau,x)$ in (\ref{march10eqn11}) and $Err_{k;j}^{ l,n}(K, f)(s,\tau,x)$ in (\ref{march10eqn41})  and the estimate of corresponding kernels in  (\ref{march10eqn10})     and (\ref{march10eqn50}).  Note that we have $s-\tau\geq \tilde{\delta}\geq 2^{-k+\epsilon M_t }$. The smallness of extra $2^{-k}$ comes from the symbol compensate the loss of smallness ``$s-\tau$'' in $Err_{k;j}^{ l,n}(K, f)(s,\tau,x)$. Therefore, with minor modification of the proof of the estimate (\ref{march11eqn61}), we have
 \be\label{march11eqn62}
 \int_0^{s-\tilde{\delta}}\big| Err_{k;j}^{ l,n}(K, f)(s,\tau,x)\big| d \tau \lesssim  C_n(x, \omega_0) \| m_1\|_{\mathcal{S}^\infty_{k,j,l}}  \| m_2\|_{\mathcal{S}^\infty_{k;n,l }}  r_{-}^{-1} 2^{k+(1+3\alpha)M_t +l/2+ n }.
 \ee
Recall the decomposition (\ref{march11eqn71}).  To sum up, our desired estimate (\ref{march10eqn2}) holds from the estimates (\ref{march10eqn15}), (\ref{march11eqn61}), and  (\ref{march11eqn62}).

Moreover,  we obtain a rough point-wise estimate   as a byproduct of the above argument, which will be used when $r$ is extremely small. Recall (\ref{march10eqn1}).  Note that after using the Cauchy-Schwarz inequality for the integration with respect to $\xi$ and the volume of support of $\xi$ and $v$, the following estimate holds
\[
|  T_{k;j}^{ l,n}(K, f)(s,\tau,x) |\lesssim   \| m_1\|_{\mathcal{S}^\infty_{k,j,l}}\| m_2\|_{\mathcal{S}^\infty_{k;n,l }} 2^{(3k+2l)/2} \int_{\R^3}\|K(\tau,x)f(\tau,x,v)\|_{L^2_x } \varphi_j(v)\psi_{\leq n+4}(\angle(-v, \omega_0))  dv
\]
\be\label{march15eqn91}
\lesssim   \| m_1\|_{\mathcal{S}^\infty_{k,j,l}}  \| m_2\|_{\mathcal{S}^\infty_{k;n,l }}  2^{3k/2+3j+2n+l }. 
\ee
Moreover, for any fixed $\tau\in[0,s]$ s.t., $|s-\tau|\geq 10 (r+2^{-k-l+\epsilon  M_t})$, we have   $|x-y-z-(s-\tau)|\sim (s-\tau)$ if   $|y|,|z|\leq 2^{-k-l+\epsilon M_t}$. Therefore, from the the rapid decay rate of kernels  in (\ref{march10eqn10}) and (\ref{march10eqn12}) if $|y|, |z|\geq 2^{-k- l+\epsilon  M_t/2 }=\tilde{\delta}2^{-\epsilon  M_t/2} $, the estimate of good part in (\ref{march11eqn59}),  the point-wise estimate of electromagnetic field (\ref{march3eqn51}) in Lemma \ref{roughestimateofelectromag}, we have
\be\label{march15eqn92}
|  T_{k;j}^{ l,n}(K, f)(s,\tau,x) |\lesssim \| m_1\|_{\mathcal{S}^\infty_{k,j,l}}  \| m_2\|_{\mathcal{S}^\infty_{k;n,l }}   \big[(s-\tau) \frac{2^{(1+\epsilon) M_t}}{|s-\tau|}   |K_n| 2^{k}   2^{3j+2l}\big( \bar{\theta}_{s-\tau}^{l}\big)^2   \big].
\ee
Combining the estimates (\ref{march15eqn91}) and (\ref{march15eqn92}), we have
\[
\int_0^s |  T_{k;j}^{ l,n}(K, f)(s,\tau,x) d \tau |\lesssim \| m_1\|_{\mathcal{S}^\infty_{k,j,l}}  \| m_2\|_{\mathcal{S}^\infty_{k;n,l }}  \big[\int_{[0,s]\cap[s-  10 (r+2^{-k-l+\epsilon M_t}), s+ 10 (r+2^{-k-l+\epsilon  M_t})]}  2^{3k/2+3j+2n+l }  d \tau
\]
\[
+ \int_{[0,s]\cap[s-  10 (r+2^{-k-l+\epsilon  M_t}),s+ 10 (r+2^{-k-l+\epsilon M_t})]^c} (s-\tau) \frac{2^{(1+\epsilon)M_t}}{|s-\tau|}   |K_n| 2^{k}   2^{3j+2l}\big( \bar{\theta}_{s-\tau}^{l}\big)^2  d \tau \big]
\]
\be
\lesssim  \| m_1\|_{\mathcal{S}^\infty_{k,j,l}}  \| m_2\|_{\mathcal{S}^\infty_{k;n,l }}  \big( 2^{3k/2+3j+2n+l +\epsilon M_t}  r + 2^{k+3j+2n+2l+(1+3\epsilon) M_t}\big). 
\ee
Hence finishing the proof of our desired estimate  (\ref{march14eqn100}). 
\end{proof}

\end{document}